\documentclass[12pt]{article}
\usepackage{amsmath,amssymb,amsthm,eucal,mathtools}
\usepackage[T1]{fontenc}
\usepackage{ tipa }
\usepackage{color}
\usepackage[all]{xy}
\usepackage[pdftex, draft=false]{hyperref} 
\usepackage{enumitem}

\title{\vspace*{-1pc}%
       Curvature and Weitzenbock formula for the  Podle\'{s} quantum sphere}
       
\author{Bram Mesland\S$^*$, Adam Rennie\dag
\thanks{email: 
\texttt{b.mesland@math.leidenuniv.nl},\ \texttt{renniea@uow.edu.au}
}
\\[3pt]
\S Mathematisch Instituut, Universiteit Leiden, Netherlands
\\[3pt]
\dag School of Mathematics and Applied Statistics, University of Wollongong\\
Wollongong, Australia
}

\advance\topmargin by -\headheight
\advance\topmargin by -\headsep
\evensidemargin=\oddsidemargin
\makeatletter
\def\section{\@startsection{section}{1}{\z@}{-3.5ex plus -1ex minus
  -.2ex}{2.3ex plus .2ex}{\large\bf}}
\def\subsection{\@startsection{subsection}{2}{\z@}{-3.25ex plus -1ex
  minus -.2ex}{1.5ex plus .2ex}{\normalsize\bf}}
\makeatother

\numberwithin{equation}{section} 

\theoremstyle{plain} 
\newtheorem{introthm}{Theorem}
\newtheorem{thm}{Theorem}[section]
\newtheorem{lemma}[thm]{Lemma}
\newtheorem{prop}[thm]{Proposition}
\newtheorem{corl}[thm]{Corollary}

\theoremstyle{definition} 
\newtheorem{defn}[thm]{Definition}
\newtheorem{example}[thm]{Example} 
\theoremstyle{remark} 
\newtheorem{rmk}[thm]{Remark}

\DeclareMathOperator{\End}{End}   
\DeclareMathOperator{\Tr}{Tr}     

\newcommand{\A}{\mathcal{A}}  
\newcommand{\B}{\mathcal{B}}  
\newcommand{\C}{\mathbb{C}}   
\newcommand{\CDB}{\mathcal{C_\D(B)}} 

\newcommand{\D}{\mathcal{D}}  
\renewcommand{\d}{\mathrm{d}_\Psi} 
\newcommand{\dee}{\mathrm{d}} 
\renewcommand{\H}{\mathcal{H}}  
\newcommand{\N}{\mathbb{N}}   
\newcommand{\nablar}{\overrightarrow{\nabla}} 
\newcommand{\nablal}{\overleftarrow{\nabla}} 
\newcommand{\ox}{\otimes}     
\newcommand{\R}{\mathbb{R}}   
\newcommand{\X}{\mathcal{X}}  

\newcommand{\ketbra}[2]{\lvert#1\rangle\langle#2\rvert} 
\newcommand{\pairing}[2]{\langle#1\mathbin{|}#2\rangle} 
\newcommand{\stroke}{\mathbin|}     
\newcommand{\alphar}{\overrightarrow{\alpha}}
\newcommand{\alphal}{\overleftarrow{\alpha}}
\newcommand{\Ch}{\textnormal{Ch}}
\def\pairL_#1(#2|#3){{}_{#1}(#2\stroke#3)} 
\def\pairR(#1|#2)_#3{(#1\stroke#2)_{#3}} 
\def\scal<#1|#2>{\langle#1\stroke#2\rangle} 

\newbox\ncintdbox \newbox\ncinttbox 
	\setbox0=\hbox{$-$}
	\setbox2=\hbox{$\displaystyle\int$}
	\setbox\ncintdbox=\hbox{\rlap{\hbox
		to \wd2{\hskip-.125em \box2\relax\hfil}}\box0\kern.1em}
	\setbox0=\hbox{$\vcenter{\hrule width 4pt}$}
	\setbox2=\hbox{$\textstyle\int$}
	\setbox\ncinttbox=\hbox{\rlap{\hbox
		to \wd2{\hskip-.175em \box2\relax\hfil}}\box0\kern.1em}

\begin{document}
\allowdisplaybreaks
\maketitle

\vspace{-2pc}

\begin{abstract}
We prove that there is a unique Levi-Civita connection on the one-forms of the Dabrowski-Sitarz spectral triple for the Podle\'{s} sphere $S^{2}_{q}$.
We compute the full curvature tensor, as well as the Ricci and scalar curvature of the Podle\'{s} sphere using the framework of \cite{MRLC}. The scalar curvature is a constant, and as the parameter $q\to 1$, the scalar curvature converges to the classical value $2$. We prove a generalised Weitzenbock formula for the spinor bundle, which differs from the classical Lichnerowicz formula for $q\neq 1$, yet recovers it for $q\to 1$.
\end{abstract}

\tableofcontents

\parskip=6pt
\parindent=0pt

\section{Introduction}

As an example of the constructions presented in \cite{MRLC}, we compute the Levi-Civita connection and curvature of the Podle\'{s} quantum sphere $S^{2}_{q}$.
The geometry of the Podle\'{s} sphere has been investigated by
numerous authors \cite{BMBook,DS,KW,Maj-pods,NT,RS,S,SW,W}. One important feature is that the classical ($q=1$) version has positive curvature, and there is no zero curvature metric, so this represents an example outside the conformally flat setting. Our results are broadly in line with those of the existing algebraic works \cite{BMBook,Maj-pods}, with the principal difference being the starting point and the techniques.

We work with the differential one-forms built  from the spectral triple \cite{DS} for the Podle\'s sphere. As is known \cite{SW,W}, this differential calculus agrees with the usual abstract differential calculus. The bimodule of one-forms in this case has trivial center, putting it outside of the scope of techniques developed in \cite{BGJ2} to construct a Levi-Civita connection for it. The general theory of \cite{MRLC, MRC} can be applied in this context and the first main result of the paper reads
\begin{introthm}
The module of differential one-forms on the Podle\'{s} sphere $S^{2}_{q}$ equipped with the quantum metric defined in Section \ref{sec: qmetric} admits a unique Levi-Civita connection whose scalar curvature equals $r=[2]_q(1+(q^{-2}-q^2)^2)$.
\end{introthm}

The relevance of this example goes beyond being able to compute junk, exterior derivative and Levi-Civita connection. Unlike classical and $\theta$-deformed manifolds, both the uniqueness of the Levi-Civita connection and the Weitzenbock formula for the Podle\'s sphere hold with a non-trivial generalised braiding. Thus  starting from connections on modules and spectral triples, we arrive at essentially the same structures as the algebraic theory, where such braidings appear naturally for other reasons.

To state a Weitzenbock formula requires a special setting.
The Dirac spectral triple of the Podle\'s sphere involves the Dirac operator $\D$  acting on a spinor module $\mathcal{S}=\mathcal{S}^{+}\oplus \mathcal{S}^{-}$. The generalised Weitzenbock formula involves both the spinor Laplacian $\Delta^{\mathcal{S}}$. Our second main result establishes a close relationship between these operators and the Clifford representation of the curvature of the spinor bundle.

\begin{introthm} The Dirac operator $\D$ and spin Laplacian $\Delta^{\mathcal{S}}$ for the Podle\'{s} sphere $S^{2}_{q}$ are related by the Weitzenbock formula
\[\D^{2}=\Delta^{\mathcal{S}}+\frac{1}{q^{-2}+q^2}\begin{pmatrix} q^{2}&0\\0&q^{-2}\end{pmatrix},\]
and the spin Laplacian $\Delta^{\mathcal{S}}$ is a positive operator.
\end{introthm}

The matrix operator appearing in the Weitzenbock formula is to be interpreted as acting on the graded module $\mathcal{S}=\mathcal{S}^{+}\oplus\mathcal{S}^{-}$. Thus, the close relationship between the curvature of the spinor module and the curvature of the module of one-forms breaks down for $q\neq 1$. Whilst a Weitzenbock formula holds, the specific Lichnerowicz relationship that 
$\D^2=\Delta+\frac{1}{4}r$ where $r$ is the scalar curvature does not hold for $q\neq 1$. That said, as $q\to 1$, we do recover the classical formula.

{\bf Acknowledgements} The authors thank the Erwin Schr\"{o}dinger 
Institute, 
Vienna, for hospitality and support during the 
production of this work. 
BM thanks the University of Wollongong for hospitality at 
an early stage of this project. AR thanks the Universiteit  Leiden for hospitality in 2022 and 2024.
The authors thank Alan Carey, Uli Kr\"{a}hmer, Giovanni Landi and Walter van Suijlekom for valuable 
discussions. 


\section{Background}
\label{sec:ass1}
We present a summary of the theoretical framework from \cite{MRLC,MRC}
which we use to define and compute the Levi-Civita connection and curvature. 

\subsection{Differentials and junk}

Throughout this article we are looking at the differential structure provided by a spectral triple. We never require a compact resolvent condition, and so have omitted it from the following, otherwise standard, definition.
See \cite[Examples 2.4--2.6]{MRLC}.
\begin{defn}
\label{defn:NDS}
Let $B$ be a $C^{*}$-algebra. A spectral triple for $B$ is a triple $(\B,\H,\D)$ where $\B\subset B$ is a local \cite[Definition 2.1]{MRLC} dense $*$-subalgebra, $\H$ is a  Hilbert space equipped with a $*$-representation $B\to \mathbb{B}(\H)$, and $\D$ an unbounded self-adjoint regular operator $\D:\,{\rm dom}(\D)\subset \H\to \H$ such that for all $a\in\B$
\[
a\cdot{\rm dom}(\D)\subset{\rm dom}(\D)\quad \mbox{and} \quad [\D,a]\ \mbox{is bounded}.
\]
\end{defn}

Given a spectral triple $(\B,\H,\D)$, the module of one-forms is the $\B$-bimodule
\[
\Omega^{1}_{\D}(\B):=\textnormal{span}\left\{a[\D,b]:a,b\in\B\right\}\subset\mathbb{B}(\H).
\]
We obtain a first order differential calculus $\mathrm{d}:\B\to \Omega^{1}_{\D}(\B)$ by setting $\mathrm{d}(b):=[\D,b]$. This calculus carries an involution $(a[\D,b])^{\dag}:=[\D,b]^{*}a^{*}$ induced by the operator adjoint. Thus $(\Omega^{1}_{\D}(\B),\dag)$ is a first order differential structure in the sense of   \cite{MRLC}.

We recollect some of the constructions of \cite{MRLC} for $(\Omega^{1}_{\D}(\B),\dag)$. Write $T^{k}_{\D}(\B):=\Omega_\D^1(\B)^{\ox_\B k}$ and $\Omega_\D^1(\B)^{k}={\rm span}\{b_0[\D,b_1]\cdots[\D,b_k]:b_j\in\B\}$. The universal differential
forms $\Omega^*_{u}(\B)$ admit  representations
\begin{align}
\pi_\D:\Omega^k_{u}(\B)\to T^{k}_{\D}(\B)\quad \pi_\D(a_0\delta(a_1)\cdots\delta(a_k))&=a_0[\D,a_1]\ox_\B\cdots\ox_\B[\D,a_k],\\
m\circ\pi_{\D}:\Omega^k_{u}(\B)\to \Omega^k_\D(\B)\quad m\circ\pi_\D(a_0\delta(a_1)\cdots\delta(a_k))&=a_0[\D,a_1]\cdots[\D,a_k],
\end{align}
where $m:T^{k}_{\D}(\B)\to\Omega^k_\D(\B)$ is the multiplication map.
Neither $\pi_\D$ nor $m\circ\pi_\D$ are maps of differential algebras, but are $\B$-bilinear maps of associative $*$-$\B$-algebras, \cite{Landi,MRLC}. The $*$-structure on $\Omega_\D^*(\B)$ is determined by the adjoint of linear maps on $\H$, while the $*$-structure on $\oplus_kT^{k}_{\D}(\B)$ is given by the operator adjoint and
\[
(\omega_1\ox_\B\omega_2\ox_\B\cdots\ox_\B\omega_k)^\dag:= 
\omega_k^*\ox_\B\cdots\ox_\B\omega_2^*\ox_\B\omega_1^*.
\]
We will write $\omega^\dag:=\omega^*$ for one forms $\omega$ as well. 
 
The maps $\pi_{\D}:\Omega^{*}_{u}(\B)\to T^*_\D$ and $\delta:\Omega^{k}_{u}(\B)\to \Omega^{k+1}_{u}(\B)$ are typically not compatible  in the sense that $\delta$ need not map $\ker \pi_{\D}$ to itself. Thus in general,  $T^{*}_{\D}(\B)$ can not be made into a differential algebra.
The issue to address is that
there are universal
forms $\omega\in \Omega^n_u(\B)$ for which $\pi_{\D}(\omega)=0$
but $\pi_{\D}(\delta(\omega))\neq0$. The latter are known as {\em junk forms},  \cite[Chapter VI]{BRB}. 
We denote the $\B$-bimodules of junk forms by
\begin{align*}
J^k_\D(\B)&=\{\pi_\D(\delta(\omega)):\,m\circ\pi_\D(\omega)=0\}\quad\mbox{and}\quad
JT^k_\D(\B)=\{\pi_\D(\delta(\omega)):\,\pi_\D(\omega)=0\}.
\end{align*}
Observe that the junk submodules only depend on the representation of the universal forms.
\begin{defn}
\label{ass:fgp-metric-junk}
A second order differential structure $(\Omega^{1}_{\D},\dag,\Psi)$ is a first order differential structure $(\Omega^{1}_{\D}(\B),\dag)$ together with an idempotent $\Psi:T^{2}_{\D}\to T^{2}_{\D}$ satisfying $\Psi\circ\dag=\dag\circ\Psi$ and $JT^2_\D(\B)\subset{\rm Im}(\Psi)\subset m^{-1}(J^2_\D(\B))$. A second order differential structure is Hermitian if
$\Omega^1_\D(\B)$ is a finitely generated projective right $\B$-module with right inner product $\pairing{\cdot}{\cdot}_\B$, such that $\Psi=\Psi^{2}=\Psi^{*}$ is a projection. 
\end{defn}

A second order differential structure admits an exterior derivative $\d:\Omega^1_\D(\B)\to T^2_\D(\B)$ 
via 
\begin{equation}
\label{eq: second-order-diff}
\d(\rho)=(1-\Psi)\circ \pi_\D\circ\delta\circ \pi_\D^{-1}(\rho).
\end{equation} 
The differential satisfies $\d([\D,b])=0$ for all $b\in\B$. A differential on one-forms allows us to define curvature for modules, and formulate torsion for connections on one-forms. 

For an Hermitian differential structure $(\Omega^{1}_{\D}(\B),\dag,\Psi,\pairing{\cdot}{\cdot})$, the module of one-forms $\Omega^1_\D(\B)$ is also a finite projective left module \cite[Lemma 2.12]{MRLC} with inner product ${}_\B\pairing{\omega}{\rho}=\pairing{\omega^\dag}{\rho^\dag}_\B$. Thus all tensor powers $T^{k}_{\D}(\B)$ carry right and left inner products. Inner products on $\Omega^k_\D(\B)$ do not arise automatically.

The two inner products on $\Omega^1_\D(\B)$ give rise to equivalent norms on $\Omega^1_\D(\B)$, and using results of \cite{KajPinWat}, $\Omega^1_\D(\B)$ is a bi-Hilbertian bimodule of finite index. To explain what this means for us, recall \cite{FL02} that a (right) frame for $\Omega^1_\D(\B)$ is a (finite) collection of elements $(\omega_j)$ that satisfy
\[
\rho=\sum_j\omega_j\pairing{\omega_j}{\rho}_\B
\]
for all $\rho\in \Omega^1_\D(\B)$. 
A finite projective bi-Hilbertian module has a  ``line element'' or ``quantum metric'' \cite{BMBook} given by 
\begin{equation}
G=\sum_j\omega_j\ox\omega_j^\dag.
\label{eq:Gee}
\end{equation} 
The line element $G$ is independent of the choice of frame, is central, meaning that $bG=Gb$ for all $b\in\B$, and
\[
{\rm span}_\B\Big\{\sum_j\omega_j\ox\omega_j^\dag:\,\mbox{for any frame }(\omega_j)\Big\}
\]
is a complemented submodule of $T^{2}_{\D}$. The endomorphisms $\alphar(G)$ and $\alphal(G)$ coincide with the identity operator on $\Omega^{1}_{\D}(\B)$. The inner product is computed  via 
\begin{equation}
-g(\omega\ox\rho):=\pairing{G}{\omega\ox\rho}_\B=\pairing{\omega^\dag}{\rho}_\B.
\label{eq:Riemann}
\end{equation}
Such bilinear inner products appear in \cite{BMBook,BGJ2,BGJ1}. The element 
\begin{equation}
e^\beta:=\sum_j{}_\B\pairing{\omega_j}{\omega_j}=-g(G)\in\B
\label{eq:ee-beta}
\end{equation} 
is independent of the choice of right frame, and
is central, positive and invertible (provided the left action of $\B$ on $\Omega^1_\D(\B)$ is faithful), and we define the normalised version $Z=e^{-\beta/2}\sum_j\omega_j\ox\omega_j^\dag$.

\subsection{Hermitian torsion-free connections}
A right connection on a right $\B$-module $\X$ is a $\C$-linear map
\[
\nablar:\X\to \X\otimes_{\B}\Omega^{1}_{\D},\quad\mbox{such that}\quad
\overrightarrow{\nabla}(x a)=\overrightarrow{\nabla}(x)a+x\ox[\D,a].
\]
There is a similar definition for left connections on left modules.
Connections always exist on finite projective modules.

Given a connection $\overrightarrow{\nabla}$ on a right inner product $\B$-module $\X$ we say that $\overrightarrow{\nabla}$ is Hermitian \cite[Definition 2.23]{MRLC} if
for all $x,y\in\X$ we have
\[
-\pairing{\overrightarrow{\nabla}x}{y}_\B+\pairing{x}{\overrightarrow{\nabla}y}_\B=[\D,\pairing{x}{y}_\B].
\]
For left connections we instead require
\[
{}_\B\pairing{\overleftarrow{\nabla}x}{y}-{}_\B\pairing{x}{\overleftarrow{\nabla}y}=[\D,{}_\B\pairing{x}{y}].
\]
If furthermore $\X$ is a $\dag$-bimodule \cite[Definition 2.8]{MRLC} such as $\X=T^k_\D$, then for each right connection $\overrightarrow{\nabla}$ on $\X$ there is a conjugate left connection $\overleftarrow{\nabla}$ given by
$\overleftarrow{\nabla}=-\dag\circ\overrightarrow{\nabla}\circ\dag$ which is Hermitian if and only if $\overrightarrow{\nabla}$ is Hermitian.
\begin{example}

Given a (right) frame $v=(x_j)\subset \X$ for a $\dag$-bimodule $\X$ we obtain left- and right Grassmann connections via
\[
\nablal^{v}(x):=[\D,{}_\B\pairing{x}{x_{j}^{\dag}}]\otimes x_{j}^{\dag},\quad \nablar^{v}(x):=x_j\ox[\D,\pairing{x_j}{x}_\B],\qquad x\in\X.
\]
The Grassmann connections are Hermitian and conjugate, that is $\nablal^{v}=-\dag\circ\nablar^{v}\circ\dag$. 
A pair of conjugate connections on $\X$ are both Hermitian if and only if for any right frame $(x_j)$ \cite[Proposition 2.30]{MRLC}
\begin{equation}
\overrightarrow{\nabla}(x_j)\ox x_j^\dag+x_j\ox\overleftarrow{\nabla}(x^\dag_j)=0.
\label{eq:conj-pair-Herm}
\end{equation}
\end{example}

The differential \eqref{eq: second-order-diff} allows us to  ask whether a connection on $\Omega^1_\D(\B)$ is torsion-free, meaning \cite[Section 4.1]{MRLC} that for any frame
\[
1\ox(1-\Psi)\big(\overrightarrow{\nabla}(\omega_j)\ox\omega_j^\dag+\omega_j\ox\d(\omega_j^\dag)\big)=0.
\]
For a Hermitian right connection, being torsion-free is equivalent to 
$(1-\Psi)\circ\overrightarrow{\nabla}=-\d$. For the conjugate left connection this becomes $(1-\Psi)\circ\overleftarrow{\nabla}=\d$, \cite[Proposition 4.5]{MRLC}.
\begin{defn} A frame $(\omega_{j})$ for $\Omega^{1}_{\D}(\B)$ is \emph{exact} if there exist $b_{j}\in\B$ such that $\omega_{j}=[\D,b_{j}]$.
\end{defn}

The existence result we will employ for Hermitian torsion-free connections is tied to the existence of an exact frame for the one-forms.

\begin{thm}
\label{thm:existence}\cite[Theorem 4.7, Corollary 4.9]{MRLC}
Let $(\Omega^1_\D,\dag,\Psi,\pairing{\cdot}{\cdot}_\B)$ be an Hermitian differential structure. Suppose there is an exact frame $v=(\omega_j)$  for $\Omega^1_\D$. Then for
the Grassmann connection $\nablar^v:\Omega^1_\D\to \Omega^1_\D\ox_\B\Omega^1_\D$ of the frame $(\omega_j)$ there is an equality of $\B$-bimodules 
\begin{equation}
JT^2_\D=\B\cdot\nablar^{v}(\dee(\B))\cdot\B.
\end{equation} 
Moreover the Grassmann connection $\nablar^{v}$ is Hermitian and torsion-free. 
\end{thm}

For uniqueness, which we discuss in the Appendix, we need  the definition of a special kind of bimodule connection.
\begin{defn}
\label{ass:ess}
Suppose that $\sigma: T^{2}_\D(\B)\to T^2_\D(\B)$ is an invertible bimodule map such that $\dag\circ\sigma=\sigma^{-1}\circ\dag$ and such that
the conjugate connections 
$\nablar,\nablal$ satisfy
\[
\sigma\circ\overrightarrow{\nabla}=\overleftarrow{\nabla}.
\]
Then we say that $\sigma$ is a braiding and that  $(\nablar,\sigma)$ is a $\dag$-bimodule connection.
\end{defn}
Under some further technical assumptions on the Hermitian differential structure \cite[Sections 4,5]{MRLC}, an Hermitian torsion-free $\sigma$-bimodule connection is unique if it exists. See Theorem \ref{thm:uniqueness} in the Appendix for details on uniqueness.

\subsection{Curvature and Weitzenbock formula}
\label{subsec:curvweitz}
Given a second order differential structure, the curvature of connections can be defined in the well-known algebraic manner. For Hermitian differential structures, can also define Ricci- and scalar curvature. Curvature and the square of the Dirac operator are related via a Weitzenbock formula on a class of Dirac spectral triples which encode the key features of Dirac bundles over manifolds.
\begin{defn}
\label{defn:Lambda}
Given a second order differential structure $(\Omega^1_{\D}(\B),\dag,\Psi)$ we set 
\[
\Lambda^2_{\D}(\B):=(1-\Psi)T^2_{\D}(\B).
\]
\end{defn}

To define curvature, we require the second order differential $\d:\Omega^1_{\D}\to \Lambda^2_{\D}$, but not any higher degree forms. For curvature, then,   the classical definition is available.

\begin{defn}
\label{def:Riemann}
If $(\Omega^1_{\D}(\B),\dag,\Psi)$ is a second order differential structure,  define the curvature of any right connection $\nablar$ on a finite projective right module $\X_\B$ to be
\[
R^{\nablar}(x)=(1\ox(1-\Psi))\circ (\nablar\ox 1+1\ox\d)\circ\nablar(x)\in \X\ox_\B\Lambda^2_{\D}(\B),\qquad x\in\X.
\]
For a connection $\nablal$ on a left module ${}_\B\X$ we define
the curvature to be
\[
R^{\nablal}(x)=((1-\Psi)\ox1)\circ (1\ox\nablal-\d\ox1)\circ\nablal(x)\in \Lambda^2_\dee(\B)\ox_\B\X,\qquad x\in\X.
\]
\end{defn}
\begin{lemma}
\label{lem:GrCurve}
If $v=(x_j)$ is a right frame for $\X$ and $\nablar^{\X}=\nablar^{v}$ the associated Grassmann connection then
\begin{align*}
R^{\nablar^\X}(x)=1\ox(1-\Psi)&\Big(x_k\ox[\D,\pairing{x_k}{x_j}_\B]\ox[\D,\pairing{x_j}{x_p}_\B]\pairing{x_p}{x}_\B\Big)
\end{align*}
Similarly if $v=(x_j)$ is a left frame for a left module $\X$ and $\nablal^{\X}=\nablal^{v}$ then
\begin{align*}
R^{\nablal^\X}(x)=(1-\Psi)\ox1&\Big({}_\B\pairing{x}{x_p}[\D,{}_\B\pairing{x_p}{x_j}]\ox[\D,{}_\B\pairing{x_j}{x_k}]\ox x_k\Big).
\end{align*}
If $\X$ is a $\dag$-bimodule, the corresponding elements in $\X\ox_{\B}\Lambda^{2}_{\D}\ox_{\B}\X$ are
\[R^{\nablar^{\X}}=x_{k}\ox (1-\Psi)\left( [\D,\pairing{x_k}{x_j}_\B]\ox[\D,\pairing{x_j}{x_p}_\B]\right)\ox x_p^{\dag},\]
and
\[R^{\nablal^\X}= x_p^{\dag}\ox (1-\Psi)\left([\D,{}_\B\pairing{x_p}{x_j}]\ox[\D,{}_\B\pairing{x_j}{x_k}]\right)\ox x_k.\]
\end{lemma}
\begin{proof}
This follows directly from \cite[Proposition 3.3]{MRC}.
\end{proof}

\begin{defn}
Let $(\Omega^{1}_{\D}(\B),\dag,\Psi,\pairing{\cdot}{\cdot})$ be a Hermitian differential structure and $\nablar$ a right connection on $\Omega^{1}_{\D}$ with curvature $R^{\nablar}\in \Omega^1_\D(\B)\ox\Lambda^2_\D(\B)\ox\Omega^1_\D(\B)$.
The Ricci curvature of $\nablar$ is
\[
{\rm Ric}^{\nablar}={}_\B\pairing{R^{\nablar}}{G}\in T^2_\D(\B)
\]
and the scalar curvature is
\[
r^{\nablar}=\pairing{G}{{\rm Ric}^{\nablar}}_\B.
\]
\end{defn}
These definitions mirror those of \cite{BMBook} and references therein, and agree when both apply.

We now recall from \cite{MRC} a class of spectral triples for which the Weitzenbock formula holds.  Given a left inner product module $\X$ and a positive functional $\phi:\B\to\C$, the Hilbert space $L^2(\X,\phi)$ is the completion of
$\X$ in the scalar product $\langle x,y\rangle:=\phi({}_\B\pairing{x}{y})$.
\begin{defn} 
\label{def: Dirac-spectral-triple}
Let $(\B,\H,\D)$ be a spectral triple equipped with a braided Hermitian differential structure $(\Omega^{1}_{\D}(\B),\dag,\Psi,\pairing{\cdot}{\cdot},\sigma)$. Then $(\B,\H,\D)$ is a \emph{Dirac spectral triple} relative to $(\Omega^{1}_{\D}(\B),\dag,\Psi,\pairing{\cdot}{\cdot},\sigma)$ 
if 
\begin{enumerate}[itemsep=0em]
\item for $\omega,\eta\in\Omega^{1}_{\D}(\B)$ we have
\begin{equation}
\label{eq:Clifford}
(m\circ\Psi)(\rho\ox\eta)=e^{-\beta}m(G)\pairing{\rho^\dag}{\eta}_\B=-e^{-\beta}m(G)g(\rho\ox\eta); 
\end{equation}
\item there is a left inner product module $\X$ over $\B$ and a positive functional $\phi:\B\to\mathbb{C}$ such that $\H=L^{2}(\X,\phi)$ and the natural map $c:\Omega^{1}_{\D}(\B)\otimes_{\B}L^{2}(\X,\phi)\to L^{2}(\X,\phi)$ restricts to a map $c:\Omega^{1}_{\D}(\B)\otimes_{\B}\X\to \X$;
\item there is a left connection $\nablal^{\X}:\X\to \Omega^{1}_{\D}(\B)\otimes_{\B}\X$ such that $\D=c\circ\nablal^{\X}:\X\to L^{2}(\X,\phi)$;
\item there is an Hermitian torsion free $\dag$-bimodule connection $(\nablar^{G},\sigma)$ on $\Omega^{1}_{\D}$ such that
\begin{align}
\D(\omega x)=c\circ\nablal^\X(c(\omega\ox x))&=c\circ (m\circ\sigma\ox 1)(\nablar^{G}(\omega)\ox x+\omega\ox\nablal^\X(x))
\label{eq:compatible2}
\end{align}
\end{enumerate}
\end{defn}
We note that for spectral triples of Dirac-type operators on Riemannian manifolds we have $e^{-\beta}m(G)={\rm Id}$, \cite[Lemma 4.4]{MRC}.
A spectral triple satisfying the conditions has a natural connection Laplacian 
$\Delta^\X$, \cite[Definition 4.3]{MRC}, defined on $x\in\X$ by
\begin{equation}
\Delta^\X(x)=e^{-\beta}m(G)\pairing{G}{(\nablar^G\ox1+1\ox\nablal^\X)\circ\nablal^\X(x)}_\B,
\label{eq:conn-lap}
\end{equation} 
and there is a Weitzenbock formula relating $\D^2$ and $\Delta^\X$.
\begin{thm}
\label{thm:Weitz1} 
Let $(\B,L^2(\X,\phi),\D)$ be a Dirac spectral triple relative to the braided Hermitian differential structure $(\Omega^{1}_{\D}(\B),\dag,\Psi,\pairing{\cdot}{\cdot},\sigma)$, and let $\Delta^{\X}$ be the connection Laplacian of the left connection $\nablal^\X$.
If $m\circ\sigma\circ\Psi=m\circ \Psi$ and $\Psi(G)=G$ then
\begin{align}
\label{eq:Weitzy}
\D^2(x)
&=\Delta^\X(x)+c\circ (m\circ\sigma\ox 1)\big(R^{\nablal^\X}(x)\big).
\end{align}
\end{thm}

In Section \ref{sec:weitzy} we will check that these conditions hold for the Podle\'{s} sphere spectral triple, and then derive the Weitzenbock formula in that case. We also use \cite[Corollary 4.9]{MRC} to show that $\Delta^\X\geq0$.

\section{The Levi-Civita connection for the Podle\'{s} sphere}
\label{sec:LC-pods}

For the remainder of this paper we will study connections on the one-forms
$\Omega^1_\D(\B)=\{\sum a^i[\D,b^i]:\,a^i,b^i\in\B\}$ of a specific spectral triple $(\B,\H,\D)$. Here $\B$ is the (coordinate algebra of) the Podle\'{s} sphere, which we describe along with the spectral triple  below. In the Appendix we will require a completion of $\B$ to use the framework of \cite{MRLC} to address the uniqueness of the algebraic Hermitian torsion-free connection we construct below.

\subsection{Coordinates, spectral triple and inner product for Podle\'s sphere}
\label{sec:pods}

We start with the polynomial algebra $\mathcal{A}:=\mathcal{O}(SU_{q}(2))$ on quantum $SU(2)$ spanned by the
matrix elements $t^l_{ij}$, with $l\in \frac{1}{2}\N$ and 
$i,\,j\in\{-l,-l+1,\dots,l-1,l\}$. While the coproduct is easy to describe in
this picture
$$
\Delta(t^l_{ij})=\sum_kt^l_{ik}\ox t^l_{kj},
$$
the product  involves (quantum) Clebsch-Gordan coefficients.
We summarise the basic algebraic relations we require using the conventions of \cite{RS}. 
The generators and relations of $\mathcal{O}(SU_q(2))$ are
\begin{equation}
\begin{array}{c}
a b=q b a, \quad a c=q c a, \quad b d=q d b, \quad c d=q d c, \quad b c=c b \\
a d=1+q b c, \quad d a=1+q^{-1} b c
\end{array}
\end{equation}
with adjoints
\begin{equation}
a^{*}=d, \quad b^{*}=-q c, \quad c^{*}=-q^{-1} b, \quad d^{*}=a
\end{equation}
The relation to the matrix elements of the defining corepresentation is 
\begin{equation}
a=t_{-\frac{1}{2},-\frac{1}{2}}^{\frac{1}{2}}, \quad b=t_{-\frac{1}{2}, \frac{1}{2}}^{\frac{1}{2}}, \quad c=t_{\frac{1}{2},-\frac{1}{2}}^{\frac{1}{2}}, \quad \quad d=t_{\frac{1}{2}, \frac{1}{2}}^{\frac{1}{2}} .
\end{equation}
The $C^{*}$-algebra  $C(SU_{q}(2))$ and polynomial algebra $\mathcal{O}(SU_{q}(2))$ carry a one parameter group of automorphisms which for $s\in\R$ is given by  $U_{s}(t^l_{ij})=q^{\sqrt{-1}sj}t^l_{ij}$. The one parameter group is periodic and thus gives an action of the circle $\mathbb{T}$. The fixed point $C^*$-algebra $C(SU_{q}(2))^{\mathbb{T}}$ for this circle action is the $C^{*}$-algebra $C(S^{2}_{q})$ of the Podle\'{s} sphere, 
$$
C(S^{2}_{q}):=C(SU_{q}(2))^{\mathbb{T}}\subset C(SU_{q}(2)).
$$ 
For future reference we recall the definition of the $q$-numbers 
$$
[x]_q=\frac{q^{-x}-q^x}{q^{-1}-q}.
$$
The Podle\'s sphere $\B:=\mathcal{O}(S^{2}_{q})\subset C(S^2_q)$ is the polynomial dense $*$-subalgebra of $C(S^{2}_{q})$ spanned by the matrix elements $t^l_{i0}$.
The generators of the Podle\'s sphere are
\begin{align*}
A&=-q^{-1}bc=c^*c=t^{1/2*}_{1/2,-1/2}t^{1/2}_{1/2,-1/2}=q^{-2}t^{1/2}_{-1/2,1/2}t^{1/2*}_{-1/2,1/2}=-q^{-1}[2]_q^{-1}t^1_{00}\\
B&=ac^*=-q^{-1}ab=t^{1/2}_{-1/2,-1/2}t^{1/2*}_{1/2,-1/2}=-q^{-1/2}[2]_q^{-1/2}t^1_{-10}\\
B^*&=cd=t^{1/2}_{1/2,-1/2}t^{1/2}_{1/2,1/2}=q^{-1/2}t^{1/2}_{1/2,-1/2}t^{1/2*}_{-1/2,-1/2}=[2]_q^{-1/2}q^{1/2}t^1_{10}.
\end{align*}
The equalities with matrix elements $t^1_{k0}$ come from Clebsch-Gordan relations summarised in \cite[Appendix A]{S}.
The generators $A,B,B^*$ obey the relations
\[
B A=q^{2} A B, \quad A B^{*}=q^{2} B^{*} A, \quad B^{*} B=A-A^{2}, \quad B B^{*}=q^{2} A-q^{4} A^{2}.
\]

The spinor module $S=S^+\oplus S^-$ is realised as the direct sum of the finitely generated projective modules $S^{\pm}:=P_{\pm}C(S^{2}_{q})^{\oplus2}$. In our conventions \cite[Section 4.3]{S} the projections $P_{\pm}$ are given by
\[
P_+=\begin{pmatrix} t^{1/2}_{1/2,1/2}t^{1/2*}_{1/2,1/2} & t^{1/2}_{1/2,1/2}t^{1/2*}_{-1/2,1/2}\\ t^{1/2}_{-1/2,1/2}t^{1/2*}_{1/2,1/2} & t^{1/2}_{-1/2,1/2}t^{1/2*}_{-1/2,1/2}\end{pmatrix}
=\begin{pmatrix} 1-A & -B^*\\ -B & q^2A\end{pmatrix},
\]
and
\[ 
P_{-}=1-P_{+}=\begin{pmatrix} t^{1/2}_{1/2,-1/2}t^{1/2*}_{1/2,-1/2} & t^{1/2}_{1/2,-1/2}t^{1/2*}_{-1/2,-1/2}\\ t^{1/2}_{-1/2,-1/2}t^{1/2*}_{1/2,-1/2} & t^{1/2}_{-1/2,-1/2}t^{1/2*}_{-1/2,-1/2} \end{pmatrix}
=\begin{pmatrix} A & B^*\\ B & 1-q^2A\end{pmatrix}.
\]
Observe that $P_++P_{-}={\rm Id}_2$ and $P_+P_{-}=0$. Hence $\End_{C(S^{2}_{q})}(S^+\oplus S^-)\cong M_2(C(S^{2}_{q}))$.

 The (smooth sections of the) spinor bundle over the Podle\'s sphere is the module
$$
\mathcal{S}=\mathcal{S}^+\oplus \mathcal{S}^-={\rm span}\{(t^l_{i,1/2}b_+,t^l_{i,-1/2}b_-)^T:\,b_\pm\in \B\},
$$
with (right) $\B$-valued inner product 
$$\pairing{(w_+,w_-)^T}{(z_+,z_-)^T}_\B=w_{+}^*z_++w_{-}^*z_-.$$ The formula for the multiplication in terms of Clebsch-Gordan coefficients shows that $\mathcal{S}$ is also a left $\B$-module.
Together with the Haar state $h:C(SU_q(2))\to\C$ we can then build
 a Hilbert space $\H=L^2(S,h)$ which carries a left representation 
 of $\B$ (see \cite{S} for formulae in our conventions). 
 
 The natural Dirac operator yielding a spectral triple $(\B,\H,\D)$ is \cite{DS,NT}
 $$
 \D=\begin{pmatrix} 0 & \partial_e\\ \partial_f & 0\end{pmatrix}
 $$
where $\partial_{e}t^l_{i,j}=\sqrt{[l+1/2]_q^2-[j+1/2]_q^2}\,\,t^l_{i,j+1}$
and $\partial_{f}t^l_{i,j}=\sqrt{[l+1/2]_q^2-[j-1/2]_q^2}\,\,t^l_{i,j-1}$.

We will abbreviate the coefficients in $\partial_e,\partial_f$ as
$
\kappa^l_k=\sqrt{[l+1/2]_q^2-[k-1/2]_q^2}.
$

With $\partial_k(t^l_{ij}):=q^{j}t^l_{ij}$, and $a,b\in \mathcal{O}(SU_q(2))$ we have
\begin{equation}
\partial_e(ab)=\partial_e(a)\partial_k(b)+\partial_k^{-1}(a)\partial_e(b)\qquad \partial_f(ab)=\partial_f(a)\partial_k(b)+\partial_k^{-1}(a)\partial_f(b).
\label{eq:bi-twisted}
\end{equation}
For $b\in\B$ we have $\partial_e\partial_f(b)=\partial_f\partial_e(b)$.
For $b\in \B$, the commutator of $\D$ with the left multiplication by $b$ is
\begin{equation}
\dee(b):=[\D,b]=\begin{pmatrix} 0 & q^{-1/2}\partial_e(b)\\ q^{1/2}\partial_f(b) & 0\end{pmatrix}.
\label{eq:dee-comm}
\end{equation}

We denote by $\CDB\subset\mathbb{B}(\H)$ the
``Clifford algebra'' generated by $\B$ 
and commutators $[\D,\B]$.
We define an operator-valued weight $\Phi:\CDB\to\B$ by
$$
\Phi(\rho)=\Tr\left(\begin{pmatrix} q & 0\\ 0 & q^{-1}\end{pmatrix} \rho\right).
$$
The functional $\Phi$ 
gives a right inner product $\pairing{\rho}{\eta}_\B:=\Phi(\rho^*\eta)$ on the Clifford algebra which restricts to one-forms
as
\begin{align}
\pairing{[\D,b_1]a_1}{[\D,b_2]a_2}_\B
&=a_1^*\Tr\left(\begin{pmatrix} q & 0\\ 0 & q^{-1}\end{pmatrix}[\D,b_1]^\dag[\D,b_2]\right)a_2\nonumber\\
&=a_1^*(q^2\partial_f(b_1)^*\partial_f(b_2)+q^{-2}\partial_e(b_1)^*\partial_e(b_2))a_2.
\label{eq:inner-prod}
\end{align}
Other metrics could be considered, and are in \cite[8.47]{BMBook}, but we will focus on this one. This metric reduces to the round metric on $S^2$ for $q=1$, and like the classical round metric, we will see that the metric \eqref{eq:inner-prod} enjoys an exact frame. 

\begin{rmk}
For $q=1$ we should consider $\Phi(T)=\frac{1}{2}\Tr(T)$ for $T\in\End^*_\B(S)$ so that $\Phi(b)=b$ for $b\in\B$. This is to access the actual (inverse) metric via $\Phi((\dee x^\mu)^*\dee x^\nu)=g^{\mu\nu}$. This means that our current definition scales the inverse metric by 2, and so the metric $g_{\mu\nu}$ by $\frac{1}{2}$. This is compatible with using the Dolbeault Laplacian $(\sqrt{2}(\partial+\bar{\partial}))^2$, \cite[Section 3.1]{H05}. The combined effect of not normalising $\Phi$ and using the Dolbealt Dirac cancel out, and we will see that the inner product on one-forms corresponds to the metric of radius 1 when $q=1$.
\end{rmk}

We recall the following facts as we will use them 
extensively in the computations to come:
\begin{equation}
(t^l_{ij})^*=(-q)^{j-i}t^l_{-i,-j}\quad\mbox{and}\quad 
t^l_{ij}=(-q)^{j-i}(t^l_{-i,-j})^*.
\label{eq:q-adjoints}
\end{equation}
\begin{equation}
\delta_{ij}=\sum_{p=-l}^l(t^l_{pi})^*t^l_{pj}=\sum_{p=-l}^lt^l_{ip}(t^l_{jp})^*\quad \mbox{Orthogonality relations.}
\label{eq:q-orthog}
\end{equation}
\begin{equation}
\kappa_1^1=\kappa^1_0=[2]_q^{1/2}=\sqrt{q+q^{-1}}\qquad \kappa^l_1=\kappa^l_0.
\label{eq:q-kappa}
\end{equation}

\begin{lemma}
Using the spectral triple $(\B,\H,\D)$, we define $\dee:\B\to\Omega^1_\D(\B)$ as $\dee(b)=[\D,b]$. Then $\Omega^1_\D(\B)$ is a first order differential structure and a $\dag$-bimodule.
\end{lemma}

\subsection{Frame and quantum metric for the Podle\'s sphere}
\label{sec: qmetric}
We now define a finite exact frame for $\Omega^1_\D(\B)$. We will use this frame in all our  computations.
\begin{lemma}
\label{lem:pods-frame}
A frame for the right $\B$-module $\Omega^1_\D(\B)$ is given by
the three module elements
\begin{align}
\omega_j&=q^{-2+j}(\kappa^1_1)^{-1}[\D,t^1_{-2+j,0}]
=q^{-2+j}\begin{pmatrix} 0 & q^{-1/2}t^1_{-2+j,1}\\ q^{1/2}t^1_{-2+j,-1} & 0\end{pmatrix}\label{eq:omega}\\
&=(-1)^{1-j}\begin{pmatrix} 0 & q^{1/2}(t^1_{2-j,-1})^*\\ q^{-1/2}(t^1_{2-j,1})^* & 0\end{pmatrix},\qquad j=1,2,3.
\label{eq:omega-alt}
\end{align}
\end{lemma}

\begin{proof}
We first compute $\omega_j^\dag$ using the formula for adjoints
\eqref{eq:q-adjoints} and the orthogonality relations \eqref{eq:q-orthog}, finding that
\begin{equation}
\omega_j^\dag=(-1)^{1-j}\begin{pmatrix} 0 & q^{-1/2}t^1_{2-j,1}\\ q^{1/2}t^1_{2-j,-1} & 0\end{pmatrix}.
\label{eq:omega-adjoint}
\end{equation}
Then for $\rho =\big(\begin{smallmatrix} 0&\rho_+\\\rho_-&0\end{smallmatrix}\big)\in\Omega^1_\D(\B)$ the orthogonality relations \eqref{eq:q-orthog} and definition of the adjoints \eqref{eq:q-adjoints} yield
\begin{align*}
&\sum_j\omega_j\pairing{\omega_j}{\rho}_\B
=\sum_j\omega_jq^{-2+j}(q^{3/2}(t^{1}_{-2+j,-1})^*\rho_-+q^{-3/2}(t^{1}_{-2+j,1})^*\rho_+)\\
&=\sum_jq^{-4+2j}\begin{pmatrix} \hspace{-60mm}0 & \hspace{-60mm}qt^1_{-2+j,1}(t^{1}_{-2+j,-1})^*\rho_-+q^{-2}t^1_{-2+j,1}(t^{1}_{-2+j,1})^*\rho_+\\
q^{2}t^1_{-2+j,-1}(t^{1}_{-2+j,-1})^*\rho_-+q^{-1}t^1_{-2+j,-1}(t^{1}_{-2+j,1})^*\rho_+&\hspace{-5mm}0
\end{pmatrix}\\
&=\sum_jq^{-4+2j}\begin{pmatrix} \hspace{-90mm}0 &  \hspace{-80mm}q^{}(-q)^{4-2j}(t^1_{2-j,-1})^*t^{1}_{2-j,1}\rho_-+q^{-2}(-q)^{6-2j}(t^{1}_{2-j,-1})^*t^{1}_{2-j,-1}\rho_+ \\
q^{2}(-q)^{2-2j}(t^{1}_{2-j,1})^*t^{1}_{2-j,1}\rho_-+q^{-1}(-q)^{4-2j}(t^{1}_{2-j,1})^*t^{1}_{2-j,-1}\rho_+&\hspace{-8mm}0
\end{pmatrix}\\
&=\sum_j\begin{pmatrix} \hspace{-60mm}0 & \hspace{-60mm}q^{}(t^1_{2-j,-1})^*t^{1}_{2-j,1}\rho_-+(t^{1}_{2-j,-1})^*t^{1}_{2-j,-1}\rho_+\\
(t^{1}_{2-j,1})^*t^{1}_{2-j,1}\rho_-+q^{-1}(t^{1}_{2-j,1})^*t^{1}_{2-j,-1}\rho_+&\hspace{-10mm}0
\end{pmatrix}\\
&=\begin{pmatrix} 0 & \rho_+\\
\rho_- & 0\end{pmatrix}=\rho. \qedhere
\end{align*}
\end{proof}

\begin{corl}
\label{cor:Wat-index}
The Watatani index $e^\beta:=\sum_j{}_\B\pairing{\omega_j}{\omega_j}$ \cite{KajPinWat} of $\Omega^1_\D(\B)$ is $e^\beta:=q^2+q^{-2}$ and the line element (or quantum metric \cite{BMBook}) is
\[
G=\sum_j\omega_j\ox\omega_j^\dag=\sum_jq\begin{pmatrix} 0 & (t^1_{2-j,-1})^*\\0&0\end{pmatrix}\ox\begin{pmatrix} 0 & 0\\t^1_{2-j,-1}&0\end{pmatrix}+q^{-1}\begin{pmatrix} 0 & 0\\(t^1_{2-j,1})^*&0\end{pmatrix}\ox\begin{pmatrix} 0 & t^1_{2-j,1}\\0&0\end{pmatrix}.
\]
Normalising gives $Z:=e^{-\beta/2}G$
and $z:=m(Z)=e^{-\beta/2}\begin{pmatrix}q &0\\ 0 & q^{-1}\end{pmatrix}$. 
\end{corl}
\begin{proof}
The first statement comes from 
Equations \eqref{eq:omega-alt} and \eqref{eq:omega-adjoint} for the frame elements, and the orthogonality relations \eqref{eq:q-orthog}. Applying the multiplication map
gives
\[
\sum_j\omega_j\omega_j^\dag
=\begin{pmatrix}q &0\\ 0 & q^{-1}\end{pmatrix}.
\]
Applying $\Phi$ gives
$
\sum_j{}_\B\pairing{\omega_j}{\omega_j}=q^2+q^{-2}.
$
\end{proof}

Observe that $\pairing{Z}{Z}_\B^{T^2_\D}=1$, so that $\ketbra{Z}{Z}$ is the orthongal projection onto $\textnormal{span } G$.

\subsection{Existence of an Hermitian torsion-free connection}
\label{sec:junk-diff}

Since our frame consists of exact one-forms (i.e. of the form $[\D,b]$ for $b\in\B$), Theorem \ref{thm:existence} applies to $\Omega^{1}_{\D}(\B)$. Hence the junk two-tensors are given by the bimodule generated by $\nablar^{G}\circ\dee(\B)$ where $\nablar^{G}$ is the Grassmann connection of the exact frame. Moreover, we know from $\nablar^{G}$ is Hermitian and in order to establish that it is also torsion-free, it suffices to show that $JT^{2}_{\D}\subset T^{2}_{\D}$ is a complemented submodule. 
Thus our immediate aim is to compute the Grassmann connection on exact forms, and determine the junk bimodule and its complement.
\begin{lemma}
The Grassmann connection $\nablar^{G}$ of the frame from Lemma \ref{lem:pods-frame} applied to an exact form $[\D,b]$ is 
\begin{align}
\nablar^{G}([\D,b])&=G\partial_e\partial_f(b)
+\sum_{j,k}\begin{pmatrix} 0 & (t^1_{2-j,-1})^*\\ 0 & 0\end{pmatrix}\ox\begin{pmatrix} 0 & t^1_{2-j,-1}(t^2_{k,-2})^*\\0&0\end{pmatrix}t^2_{k,-2}\partial_e^2(b)\nonumber\\
&\qquad\qquad+\sum_{j,k}\begin{pmatrix} 0&0\\(t^1_{2-j,1})^*&0\end{pmatrix}\ox \begin{pmatrix} 0 & 0\\ t^1_{2-j,1}(t^2_{k,2})^*&0\end{pmatrix}t^2_{k,2}\partial_f^2(b).
\label{eq:junk-decomp}
\end{align}
\end{lemma}
\begin{proof}
From the proof of Lemma \ref{lem:pods-frame} 
\[
\pairing{\omega_j}{[\D,b]}_\B
=q^{-2+j}\big(q^2t^{1*}_{-2+j,-1}\partial_f(b)+q^{-2}t^{1*}_{-2+j,1}\partial_e(b)\big),
\]
or after computing adjoints of matrix elements
\[
\pairing{\omega_j}{[\D,b]}_\B
=(-1)^{1-j}\big(qt^{1}_{2-j,1}\partial_f(b)+q^{-1}t^{1}_{2-j,-1}\partial_e(b)\big).
\]
Using Equation \eqref{eq:bi-twisted},
the relations $\partial_e(t^1_{m,1})=\partial_f(t^1_{m,-1})=0$, and the formulae for adjoints, we get
\[
\partial_e(\pairing{\omega_j}{[\D,b]}_\B)
=(-1)^{1-j}t^1_{2-j,1}\partial_e\partial_f(b)+(-1)^{1-j}\kappa^1_0t^1_{2-j,0}\partial_e(b)+(-1)^{1-j}t^1_{2-j,-1}\partial_e^2(b)
\]
and
\[
\partial_f(\pairing{\omega_j}{[\D,b]}_\B)
=(-1)^{1-j}t^1_{2-j,-1}\partial_f\partial_e(b)+(-1)^{1-j}\kappa^1_1t^1_{2-j,0}\partial_f(b)+(-1)^{1-j}t^1_{2-j,1}\partial_f^2(b).
\]
Thus
\begin{align*}
[\D,\pairing{\omega_j}{[\D,b]}_\B]&=\omega_j^*\partial_e\partial_f(b)+(-1)^{1-j}\kappa^1_1t^1_{2-j,0}[\D,b]
+\begin{pmatrix} \hspace{-30mm}0 & \hspace{-30mm}q^{-1/2}(-1)^{1-j}t^1_{2-j,-1}\partial_e^2(b)\\ q^{1/2}(-1)^{1-j}t^1_{2-j,1}\partial_f^2(b)&\hspace{-10mm}0\end{pmatrix}.
\end{align*}

Next observe that
\[
(-1)^{1-j}\omega_j=\begin{pmatrix} \hspace{-30mm}0 & \hspace{-20mm}q^{-1/2}(-1)^{1-j}q^{-2+j}t^1_{-2+j,1}\\ q^{1/2}(-1)^{1-j}q^{-2+j}t^1_{-2+j,-1} & \hspace{-20mm}0\end{pmatrix}=\begin{pmatrix} 0 & q^{1/2}t^{1*}_{2-j,-1}\\ q^{-1/2}t^{1*}_{2-j,1} & 0\end{pmatrix}
\]
and so
\[
\sum_j(-1)^{1-j}\omega_jt^1_{2-j,0}=0.
\]
As $\partial_e\partial_f(b)=\partial_f\partial_e(b)$ for $b\in\B$, we have  $\nablar^{G}([\D,b])$ is
\begin{align*}
\sum_j\omega_j\ox&[\D,\pairing{\omega_j}{[\D,b]}_\B]
=\sum_j\omega_j\ox \omega_j^*\partial_e\partial_f(b)
+\omega_j\ox \begin{pmatrix} \hspace{-30mm}0 & \hspace{-30mm}q^{-1/2}(-1)^{1-j}t^1_{2-j,-1}\partial_e^2(b)\\ q^{1/2}(-1)^{1-j}t^1_{2-j,1}\partial_f^2(b)&0\end{pmatrix}\\
&=G\partial_e\partial_f(b)
+\sum_{j,k}\begin{pmatrix} 0 & (t^1_{2-j,-1})^*\\ 0 & 0\end{pmatrix}\ox\begin{pmatrix} 0 & t^1_{2-j,-1}(t^2_{k,-2})^*\\0&0\end{pmatrix}t^2_{k,-2}\partial_e^2(b)\nonumber\\
&\qquad\qquad+\sum_{j,k}\begin{pmatrix} 0&0\\(t^1_{2-j,1})^*&0\end{pmatrix}\ox \begin{pmatrix} 0 & 0\\ t^1_{2-j,1}(t^2_{k,2})^*&0\end{pmatrix}t^2_{k,2}\partial_f^2(b).
\end{align*}
the last equality
following from the orthogonality relations, with $k=-2,-1,0,1,2$. 
\end{proof}
The two families
\[
\sum_j\begin{pmatrix} 0 & (t^1_{2-j,-1})^*\\ 0 & 0\end{pmatrix}\ox\begin{pmatrix} 0 & t^1_{2-j,-1}(t^2_{k,-2})^*\\0&0\end{pmatrix},
\qquad\sum_j\begin{pmatrix} 0&0\\(t^1_{2-j,1})^*&0\end{pmatrix}\ox \begin{pmatrix} 0 & 0\\ t^1_{2-j,1}(t^2_{k,2})^*&0\end{pmatrix}
\] 
are mutually orthogonal and are contained in $\ker(m:T^2_\D\to\Omega^2_\D(\B))$. Both families of two-tensors are orthogonal to the line element $G$ as well, and we will now use them to construct a frame for $JT^{2}_{\D}$.

We may identify $\End_\B(\mathcal{S})$ with a subset of $M_2(\A)$, where $\A\subset C(SU_{q}(2))$ is the coordinate algebra of $SU_q(2)$. This is done by identifying $\Omega^1_\D(\B)$ with off-diagonal matrices, with the $E_{12}$ component of degree 1 with respect to the circle action defining $\B$, and the $E_{21}$ component of degree -1. 
We can also identify $\Omega^1_\D\ox_\B\Omega^1_\D$
with sums of tensor products of off-diagonal matrices in $M_2(\A)\ox_\B M_2(\A)$,
but now the degrees can be $-2,0,2$. 
\begin{lemma}
\label{lem:kernel}
The $\B$-bimodules $X,Y\subset T^2_\D$ of degree $-2,2$ elements respectively have
 frames
\begin{align*}
Y_k&=q\sum_j\begin{pmatrix} 0 & (t^1_{2-j,-1})^*\\ 0 & 0\end{pmatrix}\ox\begin{pmatrix} 0 & t^1_{2-j,-1}(t^2_{k,-2})^*\\0&0\end{pmatrix}\\
X_k&=q^{-1}\sum_j\begin{pmatrix} 0&0\\(t^1_{2-j,1})^*&0\end{pmatrix}\ox \begin{pmatrix} 0 & 0\\ t^1_{2-j,1}(t^2_{k,2})^*&0\end{pmatrix}.
\end{align*}
For any two-tensor $\rho\ox\eta=(\begin{smallmatrix}0&\rho_{+}\\\rho_{-}&0\end{smallmatrix})\ox(\begin{smallmatrix}0&\eta_{+}\\\eta_{-}&0\end{smallmatrix})\in T^2_\D$ we have
\begin{align}
\pairing{Y_k}{\rho\ox\eta}_\B&=qt^2_{k,-2}\pairing{E_{12}\ox E_{12}}{\rho\ox\eta}_\B=q^{-1}t^2_{k,-2}\rho_{+}\eta_{+}
\quad\mbox{and}\quad\nonumber\\
\pairing{X_k}{\rho\ox\eta}_\B&=q^{-1}t^2_{k,2}\pairing{E_{21}\ox E_{21}}{\rho\ox\eta}_\B=qt^2_{k,2}\rho_{-}\eta_{-}
\label{eq:ex-why}
\end{align}
where $E_{ij}$ are standard matrix units, and on the right we take the inner product on $M_2(\A)\ox_\B M_2(\A)$ to be defined by Equation \eqref{eq:inner-prod}. Moreover $\pairing{Y_k}{X_l}_\B=0$ for each $k,l$. Hence the junk bimodule $JT^2_\D$ is the $\B$-span of $G$ and 
the $X_k,Y_k$.

\end{lemma}
\begin{proof}
The vectors $X_k,Y_k$ are in $T^2_\D$ by Equation \eqref{eq:junk-decomp}, and the inner product calculations \eqref{eq:ex-why} are straightforward, as is $\pairing{Y_k}{X_l}=0$. 
The Clebsch-Gordan relations tell us that $T^2_\D$
breaks up as homogenous components of degrees $-2,0,2$. Let
$\rho\ox\eta=\rho_{-}E_{21}\ox\eta_{-}E_{21}$ be an element of degree -2. Then since $ t^1_{2-j,1}\rho_{-}\in\B$ we have
\begin{align*}
\sum_{k=-2}^2X_k&\pairing{X_k}{\rho\ox\eta}_\B\\
&=q^{-2}\sum_{k=-2}^2\sum_{j=1}^3\begin{pmatrix} 0&0\\(t^1_{2-j,1})^*&0\end{pmatrix}\ox \begin{pmatrix} 0 & 0\\ t^1_{2-j,1}(t^2_{k,2})^*&0\end{pmatrix}t^2_{k,2}\pairing{E_{21}\ox E_{21}}{\rho\ox\eta}_\A\\
&=\sum_{k=-2}^2\sum_{j=1}^3\begin{pmatrix} 0&0\\(t^1_{2-j,1})^*&0\end{pmatrix}\ox \begin{pmatrix} 0 & 0\\ t^1_{2-j,1}(t^2_{k,2})^*&0\end{pmatrix}t^2_{k,2}\rho_{-}\eta_{-}\\
&=\sum_{j=1}^3\begin{pmatrix} 0&0\\(t^1_{2-j,1})^*&0\end{pmatrix}\ox \begin{pmatrix} 0 & 0\\ t^1_{2-j,1}\rho_{-}\eta_{-}&0\end{pmatrix}\\
&=\sum_{j=1}^3\begin{pmatrix} 0&0\\(t^1_{2-j,1})^*t^1_{2-j,1}\rho_{-}&0\end{pmatrix}\ox \begin{pmatrix} 0 & 0\\ \eta_{-}&0\end{pmatrix}
=\begin{pmatrix} 0&0\\ \rho_{-}&0\end{pmatrix}\ox \begin{pmatrix} 0 & 0\\ \eta_{-}&0\end{pmatrix}.
\end{align*}
A similar calculation proves the result for $Y_k$.
\end{proof}

The next Proposition identifies the orthogonal complement 
of $JT^2_\D$.

\begin{prop}
\label{prop:all-two-forms}
The two-tensors $T^2_\D$ 
can be decomposed as an orthogonal direct sum
\[
T^2_\D=X\oplus Y\oplus{\rm span}(G)\oplus{\rm span}(C)
\]
where the two-tensor $C$ satisfies that $\pairing{C}{C}_\B=\alpha$ is a scalar. The differential $\d$ is defined using the projection $1-\Psi=\frac{1}{\alpha}\ketbra{C}{C}$, and is given by
\[
\d(a[\D,b])=\frac{q}{2}C\big(q^{-1}\partial_e(a)\partial_f(b)-q\partial_f(a)\partial_e(b)\big),\qquad a,b\in\B.
\]
In particular, $(\Omega^1_\D(\B),\dag,\pairing{\cdot}{\cdot},\Psi)$ is an Hermitian differential structure.
\end{prop}
\begin{proof}
To find a non-trivial non-junk two-form, we follow Wagner \cite{W} and construct a non-trivial two-form from the twisted (by the modular group action) Chern character $\Ch_{2}(P_{+})$ of $P_+$.
According to \cite[Section 4.3]{S},
\begin{align*}
&\Ch_{2}(P_{+}) = -2 \sum_{k_{0}, k_{1}, k_{2}=0}^{1} q^{-2k_{0}}\\
& \left( t^{1/2}_{1/2-k_{0},1/2}t^{1/2\ast}_{1/2-k_{1},1/2}- \tfrac{1}{2} \delta_{k_{0}, k_{1}} \right) 
\ox t^{1/2}_{1/2-k_{1},1/2}t^{1/2\ast}_{1/2-k_{2},1/2} \ox t^{1/2}_{1/2-k_{2},1/2}t^{1/2\ast}_{1/2-k_{0},1/2}.
\end{align*}
Using the definitions we can compute
\begin{equation}
\partial_b( t^{1/2}_{1/2-k,1/2}t^{1/2\ast}_{1/2-h,1/2})=\left\{\begin{array}{ll} -q^{1/2} t^{1/2}_{1/2-k,1/2}t^{1/2\ast}_{1/2-h,-1/2} & b=e\\
 q^{-1/2}t^{1/2}_{1/2-k,-1/2}t^{1/2\ast}_{1/2-h,1/2} & b=f\end{array}\right..
\label{eq:half-derivs}
\end{equation}
A computation using the orthogonality relations \eqref{eq:q-orthog} and Equation \eqref{eq:dee-comm} shows that
\begin{align}
&\pi_\D(\Ch_2(P_+))=\sum_{k_0,k_2=0}^1q^{-2k_0}\begin{pmatrix} \hspace{-30mm}0 &\hspace{-30mm} t^{1/2}_{1/2-k_{0},1/2}t^{1/2\ast}_{1/2-k_{2},-1/2}\\t^{1/2}_{1/2-k_{0},-1/2}t^{1/2\ast}_{1/2-k_{2},1/2}&0\end{pmatrix}\ox
\begin{pmatrix} \hspace{-30mm}0 & \hspace{-30mm}-t^{1/2}_{1/2-k_{2},1/2}t^{1/2\ast}_{1/2-k_{0},-1/2}\\t^{1/2}_{1/2-k_{2},-1/2}t^{1/2\ast}_{1/2-k_{0},1/2}&0\end{pmatrix}.
\label{eq:chern}
\end{align}
Since, for example, $t^{1/2}_{1/2-k_{2},-1/2}t^{1/2\ast}_{1/2-k_{0},1/2}t^{1*}_{2-j,-1}\in\B$, we can use the orthogonality \eqref{eq:q-orthog} and adjoint \eqref{eq:q-adjoints} relations to see that
\begin{equation}
\pi_\D(\Ch_2(P_+))=\begin{pmatrix} 0 & t^{1*}_{2-j,-1}\\0&0\end{pmatrix}\ox\begin{pmatrix} 0 & 0\\ t^{1}_{2-j,-1}&0\end{pmatrix}
-q^{-2}\begin{pmatrix} 0 & 0\\ t^{1*}_{2-j,1}&0\end{pmatrix}\ox\begin{pmatrix} 0 & t^{1}_{2-j,1}\\0&0\end{pmatrix}.
\label{eq:chern-simple}
\end{equation}

To obtain a non-trivial two-form orthogonal to $Z$, we define
\begin{align*}
C&:=\pi_\D(\Ch_2(P_+))-Z\pairing{Z}{\pi_\D(\Ch_2(P_+))}_\B
=\pi_\D(\Ch_2(P_+))-\frac{1}{q^2+q^{-2}}\omega_l\ox\omega_l^\dag\pairing{\rho_{(1)}^\dag}{\rho_{(2)}}_\B,
\end{align*}
where we have written $\pi_\D(\Ch_2(P_+))=\sum\rho_{(1)}\ox\rho_{(2)}$ in Sweedler notation. Since we have
\[
\pairing{C}{C}_\B=\pairing{\rho_{(2)}}{\pairing{\rho_{(1)}}{\rho_{(1)}}_\B\rho_{(2)}}_\B-\frac{1}{q^2+q^{-2}}\pairing{\rho_{(2)}}{\rho_{(1)}^\dag}_\B\pairing{\rho_{(1)}^\dag}{\rho_{(2)}}_\B
\]
we need to compute some inner products. 
Using \eqref{eq:chern-simple} the first is given by
\begin{align*}
\pairing{\rho_{(1)}^\dag}{\rho_{(2)}}_\B
&=\Tr\left(\begin{pmatrix}q&0\\0&q^{-1}\end{pmatrix}\rho_{(1)}\rho_{(2)}\right)
=q^{-1}(q^2-q^{-2}),
\end{align*}
which yields
\begin{align}
C&=\pi_\D(\Ch_2(P_+))+q^{-1}\frac{q^{-2}-q^2}{q^{-2}+q^2}\omega_l\ox\omega_l^\dag\nonumber\\
&=\frac{2q^{-1}}{q^{-2}+q^2}\Big(q^{-1}\begin{pmatrix} 0 & t^{1*}_{2-j,-1}\\0&0\end{pmatrix}\ox\begin{pmatrix} 0 & 0\\ t^{1}_{2-j,-1}&0\end{pmatrix}-q\begin{pmatrix} 0 & 0\\ t^{1*}_{2-j,1}&0\end{pmatrix}\ox\begin{pmatrix} 0 & t^{1}_{2-j,1}\\0&0\end{pmatrix}\Big).\label{eq:volume-form}
\end{align}
The remaining inner products are computed similarly, and
\begin{equation}
\alpha=\pairing{C}{C}_\B=q^{-2}(q^2+q^{-2})-\frac{1}{q^2+q^{-2}}q^{-2}(q^2-q^{-2})^2=\frac{4q^{-2}}{q^2+q^{-2}}
\label{eq:alpha}
\end{equation}
is a scalar. Hence
\[
\frac{1}{\alpha}\ketbra{C}{C}
\]
is a rank one projection, and of course $C$ is orthogonal to $Z$. Note that $\alpha\to 2$ as $q\to 1$.

To complete the proof it suffices to show that $C$ is orthogonal to the kernel of the multiplication map. With $X_k,Y_k$ as in Lemma \ref{lem:kernel} we have
\[
\pairing{Y_k}{C}_\A=\pairing{X_k}{C}_\A=0.
\]

The formula for $\d$ follows from the definition.
\end{proof}
Since the inner product extends to the whole Clifford algebra $\CDB$, we can apply the orthogonality relations and perform the sums to see that inner products with $C$ are the same, in the sense of Lemma \ref{lem:kernel}, as inner products with 
\begin{equation}
\tilde{C}=\frac{2q^{-1}}{q^2+q^{-2}}\big(q^{-1}E_{12}\ox E_{21}-qE_{21}\ox E_{12}\big).
\label{eq:cee}
\end{equation}
Likewise inner products with $G$ are the same as inner products with
\begin{equation}
\tilde{G}=qE_{12}\ox E_{21}+q^{-1}E_{21}\ox E_{12}
\label{eq:gee}
\end{equation}
Neither $\tilde{C}$ nor $\tilde{G}$ is in $T^2_\D$, though the shorthand is useful for computations of inner products with $C$ and $G$.

Combining Theorem \ref{thm:existence} and Proposition \ref{prop:all-two-forms} proves

\begin{thm}
The right Grassmann connection $\nablar^{G}$ of the frame from Lemma \ref{lem:pods-frame} is Hermitian and torsion-free.
\end{thm}

To prove uniqueness, we need to prove concordance and $\dag$-concordance, as well as have a suitable bimodule connection. We will leave most of the uniqueness proof to the Appendix, but the braiding which yields a $\dag$-bimodule connection is presented next.

\subsection{Braiding and bimodule connection}

Next we show that the left and right  Grassmann connections are a $\dag$-bimodule connection.

\begin{defn}
\label{defn:ciggy-butt-brain}
Define $\sigma:T^2_\D\to T^2_\D$ by
\begin{align*}
\sigma(Y_k)=q^{-2}Y_k,
\quad \sigma(X_k)&=q^{2}X_k,\\
\sigma(\begin{pmatrix} 0&(t^1_{2-j,-1})^*\\0&0\end{pmatrix}\ox\begin{pmatrix} 0&0\\t^1_{2-j,-1}&0\end{pmatrix})&=q^{-2}\begin{pmatrix} 0&0\\(t^1_{2-j,1})^*&0\end{pmatrix}\ox\begin{pmatrix} 0&t^1_{2-j,1}\\0&0\end{pmatrix},\\
\sigma(\begin{pmatrix} 0&0\\(t^1_{2-j,1})^*&0\end{pmatrix}\ox\begin{pmatrix} 0&t^1_{2-j,1}\\0&0\end{pmatrix})&=q^{2}\begin{pmatrix} 0&(t^1_{2-j,-1})^*\\0&0\end{pmatrix}\ox\begin{pmatrix} 0&0\\t^1_{2-j,-1}&0\end{pmatrix},
\end{align*}
and extend as a bimodule map. Observe that $\sigma^2\neq 1$ on $X\oplus Y$.
\end{defn}
\begin{rmk}
In \cite[Lemma 3.6]{KLS} a braiding on tensor powers of holomorphic line bundles/modules on the Podle\'{s} sphere is introduced. It would be of interest to see if restricting their braiding to these particular modules coincides with the braiding $\sigma$ of Definition \ref{defn:ciggy-butt-brain}.
\end{rmk}

Using Equation \eqref{eq:junk-decomp} we determine that
\begin{align}
\nablar^{G}&([\D,b]a)=G\partial_e\partial_f(b)a+q^{-1}Y_kt^2_{2-k,-2}\partial_e^2(b)a+qX_kt^2_{2-k,2}\partial_f^2(b)a+[\D,b]\ox[\D,a]\nonumber\\
&=G\partial_e\partial_f(b)a+q^{-1}Y_kt^2_{2-k,-2}\partial_e^2(b)a+qX_kt^2_{2-k,2}\partial_f^2(b)a\\
&\qquad+q^{-2}Y_kt^2_{2-k,-2}\partial_e(b)\partial_e(a)+q^2X_kt^2_{2-k,2}\partial_f(b)\partial_f(a)\nonumber\\
&\quad+\begin{pmatrix} 0&(t^1_{2-j,-1})^*\\0&0\end{pmatrix}\ox\begin{pmatrix} 0&0\\t^1_{2-j,-1}&0\end{pmatrix}\partial_e(b)\partial_f(a)
+\begin{pmatrix} 0&0\\(t^1_{2-j,1})^*&0\end{pmatrix}\ox\begin{pmatrix} 0&t^1_{2-j,1}\\0&0\end{pmatrix}\partial_f(b)\partial_e(a)
\label{eq:right}
\end{align}
using the frame $\{X_k,Y_k,Z\}$ described in the last section. 
\begin{prop}
\label{prop:left-conn}
The conjugate left connection $\nablal^{G}:=-\dag\circ\nablar^{G}\circ\dag$ is given by
\begin{align}
\nablal^{G}(&[\D,b]a)=[\D,{}_\B\pairing{[\D,b]a}{\omega_j^*}]\ox\omega_j^*\nonumber\\
&=G\partial_f\partial_e(b)a+q^3X_kt^2_{2-k,2}\partial_f^2(b)a+q^{-3}Y_kt^2_{2-k,-2}\partial_e^2(b)a+q^4X_kt^2_{2-k,2}\partial_f(b)\partial_f(a)\nonumber\\
&
+q^{-4}Y_kt^2_{2-k,-2}\partial_e(b)\partial_e(a)
+q^2\begin{pmatrix} 0&(t^1_{2-j,-1})^*\\0&0\end{pmatrix}\ox\begin{pmatrix} 0&0\\t^1_{2-j,-1}&0\end{pmatrix}\partial_f(b)\partial_e(a)\nonumber\\
&+q^{-2}\begin{pmatrix} 0&0\\(t^1_{2-j,1})^*&0\end{pmatrix}\ox\begin{pmatrix} 0&t^1_{2-j,1}\\0&0\end{pmatrix}\partial_e(b)\partial_f(a).
\label{eq:left}
\end{align}
\end{prop}
\begin{proof}
We start with computing $\nablal^{G}$ on exact forms. First, the definition yields
\[
{}_\B\pairing{[\D,b]}{\omega_j^\dag}=(-1)^{1-j}\big(\partial_e(b)(t^1_{2-j,1})^*+\partial_f(b)(t^1_{2-j,-1})^*\big).
\]
Next using the twisted derivation rule \eqref{eq:bi-twisted} and the orthogonality relations \eqref{eq:q-orthog}  in the same way as in Section \ref{sec:junk-diff} we find
\begin{align*}
&[\D,{}_\B\pairing{[\D,b]}{\omega_j^\dag}]\ox\omega_j^\dag\\
&=\begin{pmatrix} 0 & q^{-3/2}\partial_e^2(b)\\q^{-1/2}\partial_f\partial_e(b)&0\end{pmatrix}\ox\begin{pmatrix}0 & q^{-1/2}\\ 0 & 0\end{pmatrix}
+\begin{pmatrix} 0 & q^{1/2}\partial_e\partial_f(b)\\q^{3/2}\partial_f^2(b)&0\end{pmatrix}\ox\begin{pmatrix}0 & 0\\ q^{1/2} & 0\end{pmatrix}\\
&=G\partial_f\partial_e(b)+q^3X_kt^2_{2-k,2}\partial_f^2(b)+q^{-3}Y_kt^2_{2-k,-2}\partial_e^2(b),
\end{align*}
where we also used $\partial_e\partial_f(b)=\partial_f\partial_e(b)$ for $b\in\B$. Now use the Leibniz rule to find
\begin{align*}
\nablal^{G}([\D,b]a)&=\nablal^{G}([\D,ba])-\nablal^{G}(b[\D,a])\\
&=\nablal^{G}([\D,ba])-b\nablal^{G}([\D,a])-[\D,b]\ox[\D,a]\\
&=G\big(\partial_f\partial_e(ba)-b\partial_f\partial_e(a)\big)+q^3X_kt^2_{2-k,2}\big(\partial_f^2(ba)-b\partial_f^2(a)\big)\\
&+q^{-3}Y_kt^2_{2-k,-2}\big(\partial_e^2(ba)-b\partial_e^2(a)\big)
-[\D,b]\ox[\D,a].
\end{align*}
Using the twisted derivation rule \eqref{eq:bi-twisted} yields the three relations
\begin{align*}
\partial_f\partial_e(ba)-b\partial_f\partial_e(a)&=
\partial_f\partial_e(b)a+q\partial_f(b)\partial_e(a)
+q^{-1}\partial_e(b)\partial_f(a)\\
\partial_f^2(ba)-b\partial_f^2(a)
&=\partial_f^2(b)a+(q+q^{-1})\partial_f(b)\partial_f(a)\\
\partial_e^2(ba)-b\partial_e^2(a)
&=\partial_e^2(b)a+(q+q^{-1})\partial_e(b)\partial_e(a).
\end{align*}
Together with expanding the $[\D,b]\ox[\D,a]$ term
\begin{align*}
&[\D,b]\ox[\D,a]=q^{-2}Y_kt^2_{2-k,-2}\partial_e(b)\partial_e(a)
+q^{2}X_kt^2_{2-k,2}\partial_f(b)\partial_f(a)\\
&+\begin{pmatrix} 0&(t^1_{2-j,-1})^*\\0&0\end{pmatrix}\ox\begin{pmatrix} 0&0\\t^1_{2-j,-1}&0\end{pmatrix}\partial_e(b)\partial_f(a)
+\begin{pmatrix} 0&0\\(t^1_{2-j,1})^*&0\end{pmatrix}\ox\begin{pmatrix} 0&t^1_{2-j,1}\\0&0\end{pmatrix}\partial_f(b)\partial_e(a),
\end{align*}
we find that
\begin{align*}
\nablal^{G}&([\D,b]a)
=G\partial_f\partial_e(b)a+q^3X_kt^2_{2-k,2}\partial_f^2(b)a+q^{-3}Y_kt^2_{2-k,-2}\partial_e^2(b)a\\
&+G(q\partial_f(b)\partial_e(a)+q^{-1}\partial_e(b)\partial_f(a))+q^3X_kt^2_{2-k,2}(q+q^{-1})\partial_f(b)\partial_f(a)\\
&+q^{-3}Y_kt^2_{2-k,-2}(q+q^{-1})\partial_e(b)\partial_e(a)
-q^{-2}Y_kt^2_{2-k,-2}\partial_e(b)\partial_e(a)
-q^{2}X_kt^2_{2-k,2}\partial_f(b)\partial_f(a)\\
&-\begin{pmatrix} 0&(t^1_{2-j,-1})^*\\0&0\end{pmatrix}\ox\begin{pmatrix} 0&0\\t^1_{2-j,-1}&0\end{pmatrix}\partial_e(b)\partial_f(a)
-\begin{pmatrix} 0&0\\(t^1_{2-j,1})^*&0\end{pmatrix}\ox\begin{pmatrix} 0&t^1_{2-j,1}\\0&0\end{pmatrix}\partial_f(b)\partial_e(a)\\
&=G\partial_f\partial_e(b)a+q^3X_kt^2_{2-k,2}\partial_f^2(b)a+q^{-3}Y_kt^2_{2-k,-2}\partial_e^2(b)a+q^4X_kt^2_{2-k,2}\partial_f(b)\partial_f(a)\\
&+q^{-4}Y_kt^2_{2-k,-2}\partial_e(b)\partial_e(a)
+q^2\begin{pmatrix} 0&(t^1_{2-j,-1})^*\\0&0\end{pmatrix}\ox\begin{pmatrix} 0&0\\t^1_{2-j,-1}&0\end{pmatrix}\partial_f(b)\partial_e(a)\\
&
+q^{-2}\begin{pmatrix} 0&0\\(t^1_{2-j,1})^*&0\end{pmatrix}\ox\begin{pmatrix} 0&t^1_{2-j,1}\\0&0\end{pmatrix}\partial_e(b)\partial_f(a).
\end{align*}
This completes the proof.
\end{proof}
Comparing Equations \eqref{eq:right} and \eqref{eq:left} we obtain
\begin{prop}
\label{prop:call-this-a-prop?}
The conjugate pair of connections $\nablar^{G}$ and $\nablal^{G}$ satisfy $\sigma\circ \nablar^{G}=\nablal^{G}$.
\end{prop}

\begin{rmk}
\label{rmk:sigma}
It is worth noting that $\sigma(G)=G$, but 
\begin{equation}
\sigma(C)=-C+2q^{-1}\frac{q^{-2}-q^2}{q^{-2}+q^2}G.
\label{eq:sigma-cee}
\end{equation}

Since we only require that $(1-\Psi)\sigma\nablar^{G}=-(1-\Psi)\nablar^{G}$ for compatibility with the torsion-free condition, to see that \eqref{eq:sigma-cee} is compatible with the torsion-free condition, we need only check that $\frac{1}{\alpha}\ketbra{C}{C}\sigma(C)=-C$. This in turn follows from the definition of the inner product, and the orthogonality of $C$ and $G$. Moreover on ${\rm span}(C,G)$ we have $\sigma^2={\rm Id}$.
\end{rmk}
In Theorem \ref{thm:unique-pods} of the Appendix, we prove that in fact the pair $(\nablar^{G},\sigma)$ is the unique Hermitian torsion-free $\sigma$-$\dag$-bimodule connection.

\section{Curvature of the Podle\'s sphere}

We now have all the requisite structure to compute the curvature of the unique Hermitian torsion-free connection $\nablar^{G}$ on the
Podle\'s sphere. 

\subsection{The Riemann tensor}
We first compute the full Riemann tensor following Definition \ref{def:Riemann} in Section \ref{subsec:curvweitz}. Since we are dealing with a Grassmann connection,  we will use the explicit formulae of Lemma \ref{lem:GrCurve}.
\begin{thm}
The Riemann tensor of the Podle\'{s} sphere is
\begin{align*}
R=\frac{[2]_q}{2}q\,\omega_i\ox C\ox \begin{pmatrix} q^{-2}&0\\0&-q^2\end{pmatrix}\omega_i^\dag=\frac{[2]_q}{2}q\,\omega_i\ox C\ox \omega_i^\dag\begin{pmatrix} -q^{2}&0\\0&q^{-2}\end{pmatrix}.
\end{align*}
\end{thm}
\begin{proof}
We start by computing inner products of frame elements and their
differentials. First
\begin{align}
&\pairing{\omega_j}{\omega_k}_\B\nonumber\\
&=q^{-4+j+k}\Tr\left(\begin{pmatrix} q & 0\\ 0 & q^{-1}\end{pmatrix}\begin{pmatrix} 0 & q^{1/2}t^{1*}_{-2+j,-1}\\ q^{-1/2}t^{1*}_{-2+j,1} & 0\end{pmatrix}\begin{pmatrix} 0 & q^{-1/2}t^1_{-2+k,1}\\ q^{1/2}t^1_{-2+k,-1} & 0\end{pmatrix}\right)\nonumber\\
&=q^{-4+j+k}(q^2t^{1*}_{-2+j,-1}t^1_{-2+k,-1}+q^{-2}t^{1*}_{-2+j,1}t^1_{-2+k,1})\nonumber\\
&=(-1)^{1-j}q^{-1+k}t^{1}_{2-j,1}t^1_{-2+k,-1}+(-1)^{1-j}q^{-3+k}t^{1}_{2-j,-1}t^1_{-2+k,1}.
\label{eq:pods-frame-ips}
\end{align}
Then we apply the twisted Leibniz rule and $\partial_e(t^1_{x,1})=\partial_f(t^1_{x,-1})=0$ to find
\begin{align*}
&[\D,\pairing{\omega_j}{\omega_k}_\B]
=(-1)^{1-j}q^{-1+k}[\D,t^{1}_{2-j,1}t^1_{-2+k,-1}]+(-1)^{1-j}q^{-3+k}[\D,t^{1}_{2-j,-1}t^1_{-2+k,1}]\\
&\qquad\qquad=(-1)^{1-j}q^{-1+k}\begin{pmatrix} 0 & q^{-1/2}\partial_e(t^{1}_{2-j,1}t^1_{-2+k,-1})\\
q^{1/2}\partial_f(t^{1}_{2-j,1}t^1_{-2+k,-1}) & 0\end{pmatrix}\\
&\qquad\qquad\qquad\qquad+(-1)^{1-j}q^{-3+k}\begin{pmatrix} 0 & q^{-1/2}\partial_e(t^{1}_{2-j,-1}t^1_{-2+k,1})\\
q^{1/2}\partial_f(t^{1}_{2-j,-1}t^1_{-2+k,1})&0\end{pmatrix}\\
&=(-1)^{1-j}\begin{pmatrix} \hspace{-20mm}0 & \hspace{-20mm}q^{-5/2+k}\kappa^1_0t^{1}_{2-j,1}t^1_{-2+k,0}\\
q^{-3/2+k}\kappa^1_1t^{1}_{2-j,0}t^1_{-2+k,-1} & 0\end{pmatrix}\\
&\qquad\qquad+(-1)^{1-j}\begin{pmatrix} \hspace{-20mm}0 & \hspace{-20mm}q^{-5/2+k}\kappa^1_0t^{1}_{2-j,0}t^1_{-2+k,1}\\
q^{-3/2+k}\kappa^1_1t^{1}_{2-j,-1}t^1_{-2+k,0}&0\end{pmatrix}\\
&=(-1)^{1-j}\kappa^1_1q^{-2+k}\begin{pmatrix} \hspace{-10mm}0 & \hspace{-20mm}q^{-1/2}(t^{1}_{2-j,1}t^1_{-2+k,0}+t^{1}_{2-j,0}t^1_{-2+k,1})\\
\hspace{-0mm}q^{1/2}(t^{1}_{2-j,0}t^1_{-2+k,-1}+t^{1}_{2-j,-1}t^1_{-2+k,0}) & \hspace{-0mm}0\end{pmatrix}\\
&=(-1)^{1-j}\kappa^1_1t^1_{2-j,0}\omega_k+\kappa^1_1q^{-2+k}\omega_j^\dag t^1_{-2+k,0}.
\end{align*}
Thus we find that 
\begin{align*}
&[\D,\pairing{\omega_i}{\omega_j}_\B]\ox[\D,\pairing{\omega_j}{\omega_k}_\B]\\
&=(\kappa^1_1)^2\Big((-1)^{1-i}t^1_{2-i,0}\omega_j+q^{-2+j}\omega_i^\dag t^1_{-2+j,0}\Big)\ox\Big((-1)^{1-j}t^1_{2-j,0}\omega_k+q^{-2+k}\omega_j^\dag t^1_{-2+k,0}\Big)\\
&=(\kappa^1_1)^2\Big((-1)^{i+j}t^1_{2-i,0}\omega_j\ox t^1_{2-j,0}\omega_k+(-1)^{i}t^1_{2-i,0}\omega_j\ox q^{-2+k}\omega_j^\dag t^1_{-2+k,0}\\
&\qquad\qquad\qquad+(-1)^{j}q^{-2+j}\omega_i^\dag t^1_{-2+j,0}\ox t^1_{2-j,0}\omega_k
+q^{-4+j+k}\omega_i^\dag t^1_{-2+j,0}\ox \omega_j^\dag t^1_{-2+k,0}\Big).
\end{align*}
Summing over $j$, using the $\B$-centrality of $G=\sum_j\omega_j\ox\omega_j^\dag$ and Equation \eqref{eq:omega-alt} gives
\begin{align*}
&\sum_j[\D,\pairing{\omega_i}{\omega_j}_\B]\ox[\D,\pairing{\omega_j}{\omega_k}_\B]\\
&=(\kappa^1_1)^2\Big(\sum_j(-1)^{i+j}t^1_{2-i,0}\omega_jt^1_{2-j,0}\ox \omega_k+(-1)^{i}q^{-2+k}t^1_{2-i,0}t^1_{-2+k,0}G\\
&\qquad+\omega_i^\dag \ox\omega_k
+q^{-2+k}\omega_i^\dag\ox \sum_j(-1)^j(\omega_jt^1_{2-j,0})^\dag t^1_{-2+k,0}\Big)\\
&=(\kappa^1_1)^2\Big((-1)^{i}q^{-2+k}t^1_{2-i,0}t^1_{-2+k,0}G+\omega_i^\dag\ox\omega_k\Big),
\end{align*}
where the explicit form of $\omega_j$ and the orthogonality relations Equation \eqref{eq:q-orthog} give the vanishing of the two terms. 
From Equation \eqref{eq:pods-frame-ips} we also have
\[
\pairing{\omega_i}{\omega_k}_\B=\delta_{ik}-(-1)^{i+k}t^1_{2-i,0}t^{1*}_{2-k,0}
\]
so
\[
\sum_j[\D,\pairing{\omega_i}{\omega_j}_\B]\ox[\D,\pairing{\omega_j}{\omega_k}_\B]=(\kappa^1_1)^2(\delta_{ik}\sum_{l}\omega_l\ox\omega_l^\dag-\pairing{\omega_i}{\omega_k}_\B\sum_{l}\omega_l\ox\omega_l^\dag+\omega_i^\dag\ox\omega_k),
\]
and writing
\begin{align*}
\tilde{R}:&=\omega_i\ox\sum_j[\D,\pairing{\omega_i}{\omega_j}_\B]\ox[\D,\pairing{\omega_j}{\omega_k}_\B]\ox\omega_k^\dag\\
&=(\kappa^1_1)^2(\delta_{ik}\omega_i\ox\sum_{l}\omega_l\ox\omega_l^\dag-\omega_k\ox\sum_{l}\omega_l\ox\omega_l^\dag+\omega_i\ox\omega_i^\dag\ox\omega_k)\ox\omega_k^\dag\\
&=(\kappa^1_1)^2(\omega_k\ox\sum_{l}\omega_l\ox\omega_l^\dag-\omega_k\ox\sum_{l}\omega_l\ox\omega_l^\dag+\omega_i\ox\omega_i^\dag\ox\omega_k)\ox\omega_k^\dag\\
&=(\kappa^1_1)^2\omega_i\ox\omega_i^\dag\ox\omega_k\ox\omega_k^\dag=[2]_{q}\omega_i\ox\omega_i^\dag\ox\omega_k\ox\omega_k^\dag,
\end{align*}
we have $R=(1\ox (1-\Psi)\ox 1 )(\tilde{R})$. 
to obtain the curvature tensor 
\begin{align*}
R=\frac{[2]_q}{\alpha}\omega_i\ox C\pairing{C}{\omega_i^\dag\ox\omega_k}_\B\ox\omega_k^\dag
&=\frac{[2]_q}{\alpha}\,\omega_i\ox C\ox \big(\pairing{\omega_i^\dag}{C_{(1)}}_\B C_{(2)}\big)^\dag\\
&=\frac{[2]_q}{\alpha}\,\omega_i\ox C\ox C_{(2)}^\dag \pairing{C_{(1)}}{\omega_i^\dag}_\B\\
&=\frac{[2]_q}{2}q\,\omega_i\ox C\ox \begin{pmatrix} q^{-2}&0\\0&-q^2\end{pmatrix}\omega_i^\dag.\qedhere
\end{align*}
\end{proof}

\subsection{The Ricci and scalar curvature}
We now have the pieces in place to compute the Ricci and scalar curvature of $S^{2}_{q}$.
\begin{prop}
The Ricci curvature is given by 
\[
{\rm Ric}={}_\B\pairing{R}{G}=\frac{[2]_q}{q^2+q^{-2}}\,\omega_i\ox\begin{pmatrix}q^{-4}&0\\0&q^4\end{pmatrix}\omega_i^\dag=\frac{[2]_q}{q^2+q^{-2}}\,\omega_i\ox\omega_i^\dag\begin{pmatrix} q^{4} & 0\\0&q^{-4}\end{pmatrix}.
\]
\end{prop}
\begin{proof}
We start by expanding the inner product ${\rm Ric}={}_\B\pairing{R}{G}$
defining the Ricci curvature to obtain
\begin{align*}
{\rm Ric}&=\frac{[2]_q}{2}q\,\omega_i\ox C_{(1)}{}_\B\pairing{C_{(2)}\ox\begin{pmatrix}q^{-2}&0\\0&-q^2\end{pmatrix}\omega_i^\dag}{qE_{12}\ox E_{21}+q^{-1}E_{21}\ox E_{12}}\\
&=\frac{[2]_q}{q^2+q^{-2}}\,\omega_i\Big(q^{-1}E_{12}\, {}_\B\pairing{E_{21}\,{}_\B\pairing{\begin{pmatrix}q^{-2}&0\\0&-q^2\end{pmatrix}\omega_i^\dag}{qE_{21}}}{E_{12}}\\
&+q^{-1}E_{12}\,{}_\B\pairing{E_{21}\,{}_\B\pairing{\begin{pmatrix}q^{-2}&0\\0&-q^2\end{pmatrix}\omega_i^\dag}{q^{-1}E_{12}}}{E_{21}}\\
&-qE_{21}\,{}_\B\pairing{E_{12}\,{}_\B\pairing{\begin{pmatrix}q^{-2}&0\\0&-q^2\end{pmatrix}\omega_i^\dag}{qE_{21}}}{E_{12}}\\
&-qE_{21}\,{}_\B\pairing{E_{12}\,{}_\B\pairing{\begin{pmatrix}q^{-2}&0\\0&-q^2\end{pmatrix}\omega_i^\dag}{q^{-1}E_{12}}}{E_{21}}\Big).
\end{align*}
where we have abbreviated the computation by taking inner products using
\[
\tilde{G}=qE_{12}\ox E_{21}+q^{-1}E_{21}\ox E_{12}.
\]
First, two of the four terms in our expression for ${\rm Ric}$ are zero as $E_{12}$ and $E_{21}$ are orthogonal for the left inner product as well.
Next
\begin{align*}
&{}_\B\pairing{\begin{pmatrix}q^{-2}&0\\0&-q^2\end{pmatrix}\omega_i^\dag}{E_{21}}=(-1)^iq^{3/2}t^1_{2-i,-1}\\
&{}_\B\pairing{\begin{pmatrix}q^{-2}&0\\0&-q^2\end{pmatrix}\omega_i^\dag}{E_{12}}=(-1)^{i+1}q^{-3/2}t^1_{2-i,1}
\end{align*}
and so
\begin{align*}
{\rm Ric}&\negthinspace=\negthinspace\frac{[2]_q}{q^2+q^{-2}}\,\omega_i\ox\negthinspace\Big((-1)^{1+i}q^{-7/2}t^1_{2-i,1}E_{12}\,{}_\B\pairing{E_{21}}{E_{21}}\negthinspace+\negthinspace(-1)^{i+1}q^{7/2}t^1_{2-i,-1}E_{21}\,{}_\B\pairing{E_{12}}{E_{12}}\Big).
\end{align*}
Since
\[
{}_\B\pairing{E_{12}}{E_{12}}=q,\qquad {}_\B\pairing{E_{21}}{E_{21}}=q^{-1}
\]
we find that
\begin{align*}
{\rm Ric}&=\frac{[2]_q}{q^2+q^{-2}}\,\omega_i\ox\Big(q^{-9/2}(-1)^{1+i}t^1_{2-i,1}E_{12}+q^{9/2}(-1)^{1+i}t^1_{2-i,-1}E_{21}\Big)\\
&=\frac{[2]_q}{q^2+q^{-2}}\,\omega_i\ox\begin{pmatrix} q^{-4} & 0\\0&q^4\end{pmatrix}\omega_i^\dag=\frac{[2]_q}{q^2+q^{-2}}\,\omega_i\ox\omega_i^\dag\begin{pmatrix} q^{4} & 0\\0&q^{-4}\end{pmatrix}.\qedhere
\end{align*}
\end{proof}
Observe that if $q=1$ then the Ricci curvature is proportional to the line element $G=\sum_{i}\omega_i\ox\omega_i^\dag$. When $q\neq 1$, this relation breaks down and the metric on the Podle\'{s} sphere is not Einstein in the classical sense.

\begin{prop}
The scalar curvature $r=\pairing{G}{{\rm Ric}}_\B$ is constant, given by
\[
r=\frac{[2]_q}{q^2+q^{-2}}(q^{-6}+q^6)=[2]_q(1+(q^{-2}-q^2)^2),
\]
and as $q\to 1$ the scalar curvature converges to 2.
\end{prop}
\begin{proof}
Again, this is a computation, with
\begin{align*}
r&=\frac{[2]_q}{q^2+q^{-2}}\Big(q\pairing{E_{21}}{\pairing{E_{12}}{\omega_i}_\B\begin{pmatrix} q^{-4} & 0\\0&q^4\end{pmatrix}\omega_i^\dag}_\B
+\pairing{E_{12}}{\pairing{E_{21}}{\omega_i}_\B\begin{pmatrix} q^{-4} & 0\\0&q^4\end{pmatrix}\omega_i^\dag}_\B\Big)\\
&=\frac{[2]_q}{q^2+q^{-2}}q\pairing{E_{21}}{q^{-1/2}(-1)^{1+i}(t^1_{2-i,-1})^*\begin{pmatrix} q^{-4} & 0\\0&q^4\end{pmatrix}\omega_i^\dag}_\B\\
&\qquad\qquad+\frac{[2]_q}{q^2+q^{-2}}q^{-1}\pairing{E_{12}}{q^{1/2}(-1)^{1+i}(t^1_{2-i,1})^*\begin{pmatrix} q^{-4} & 0\\0&q^4\end{pmatrix}\omega_i^\dag}_\B\\
&=\frac{[2]_q}{q^2+q^{-2}}\Big(q^6(t^1_{2-i,-1})^*t^1_{2-i,-1}+q^{-6}(t^1_{2-i,1})^*t^1_{2-i,1}\Big)\\
&=\frac{[2]_q}{q^2+q^{-2}}\big(q^6+q^{-6}\big).
\end{align*}
The last expression follows from $(q^{-6}+q^{6})(q^{-2}+q^{2})^{-1}=(q^{-4}+q^{4})-1=(q^{-2}-q^2)^2+1$.
\end{proof}

{\bf Question} Is there an inner product on the one-forms for which the scalar curvature is precisely $[2]_q$?

\section{Weitzenbock formula}
\label{sec:weitzy}
\subsection{Dirac spectral triple}
\label{subsec:Dirac}
We will establish that $(\B,L^2(S^+\oplus S^-,h),\D)$ is a Dirac spectral triple by verifying conditions 1-4 of Definition \ref{def: Dirac-spectral-triple}. The first of these conditions is straightforward.
\begin{lemma} For $\rho\otimes \eta\in JT^{2}_{\D}$ we have the equality
\[m\circ\Psi(\rho\otimes \eta)=e^{-\beta}m(G)\pairing{\rho^{\dag}}{\eta}_{\B}.\]
Moreover we have $\sigma(G)=G$ and $m\circ\sigma\circ\Psi=m\circ\Psi$.
\end{lemma}
\begin{proof} We have that $\Psi T^{2}_{\D}=X\oplus Y\oplus \textnormal{span}(G)$, and $X\oplus Y\subset \ker m$ whereas the projection onto $\textnormal{span}(G)$ is given by $e^{-\beta}|G\rangle\langle G|$. Therefore
\begin{align*}
m\circ\Psi(\rho\ox \eta)=m\left(e^{-\beta}G\pairing{G}{\rho\ox\eta}_{\B}\right)=e^{-\beta}m(G)\pairing{G}{\rho\ox\eta}_{B}=e^{-\beta}m(G)\pairing{\rho^{\dag}}{\eta},
\end{align*}
as claimed. The equality $\sigma(G)=G$ was noted in Remark \ref{rmk:sigma}. By definition of $\sigma$, we have $\sigma(X\oplus Y)=X\oplus Y$. Therefore $m\circ\sigma\circ\Psi=m\circ\Psi$.
\end{proof}

We give the spinor bundles $S^\pm$ the left inner products 
\[
{}_\B\pairing{s_+}{t_+}=\Phi(\ketbra{s_+}{t_+})=qs_+t_+^*,\quad
{}_\B\pairing{s_-}{t_-}=\Phi(\ketbra{s_-}{t_-})=q^{-1}s_-t_-^*,
\]
and observe that they restrict to $\B$-valued inner products on $\mathcal{S}^\pm$.
Now let $\nablal^{\mathcal{S}^{\pm}}$ be the left Grassmann connections of the bundles $S^\pm$ for the left frames 
$s_{\pm1/2,-}:=q^{1/2}t^{1/2}_{\pm1/2,-1/2}$, $s_{\pm1/2,+}:=q^{-1/2}t^{1/2}_{\pm1/2,1/2}$, and $\nabla^{\mathcal{S}}:=\nabla^{\mathcal{S}^{+}}\oplus\nabla^{\mathcal{S}^{-}}$. Moreover the action of the Clifford algebra $\CDB$ on $L^{2}(S,h)$ restricts to an action 
\[
c:\CDB\ox_\B(\mathcal{S}^+\oplus \mathcal{S}^-)\to \mathcal{S}^+\oplus \mathcal{S}^-,
\] 
of the Clifford algebra on spinors, which establishes condition 2 of Definition \ref{def: Dirac-spectral-triple}. We now verify condition 3 of Definition \ref{def: Dirac-spectral-triple}.
\begin{lemma} The Dirac operator $\D$ satisfies $\D=c\circ\nabla^{\mathcal{S}}: \mathcal{S}\to L^{2}(S,h)$.
\end{lemma}
\begin{proof}
As Grassmann connections are always Hermitian, we use
\cite[Proposition 2.24]{MRLC} in the computation 
\begin{align*}
\D&\begin{pmatrix}s\\0\end{pmatrix}
=\sum_i[\D,{}_\B\pairing{s}{s_{i,+}}]\begin{pmatrix}s_{i,+}\\0\end{pmatrix}+{}_\B\pairing{s}{s_{i,+}}\D\begin{pmatrix}s_{i,+}\\0\end{pmatrix}\\
&=\sum_ic\Big({}_\B\pairing{\nablal^{\mathcal{S}^+}(s)}{s_{i,+}}\begin{pmatrix}s_{i,+}\\0\end{pmatrix}\negthinspace\Big)\negthinspace-c\Big({}_\B\pairing{s}{\nablal^{\mathcal{S}^+}(s_{i,+})}\begin{pmatrix}s_{i,+}\\0\end{pmatrix}\negthinspace\Big)\negthinspace+{}_\B\pairing{s}{s_{i,+}}\D\begin{pmatrix}s_{i,+}\\0\end{pmatrix}\\
&=\sum_ic\circ\nablal^{\mathcal{S}^+}\begin{pmatrix}s\\0\end{pmatrix}
-c\Big({}_\B\pairing{s}{s_{i,+}}\begin{pmatrix}\nablal^{\mathcal{S}^+}s_{i,+}\\0\end{pmatrix}\Big)+{}_\B\pairing{s}{s_{i,+}}\D\begin{pmatrix}s_{i,+}\\0\end{pmatrix}\\
&=c\circ\nablal^{\mathcal{S}^+}\begin{pmatrix}s\\0\end{pmatrix}
\end{align*}
and one can check the final equality by checking that $c\circ\nablal^{\mathcal{S}}$ and $\D$ agree on frame elements using the formulae \eqref{eq:derivs}.
The same argument holds for $S^-$ as well, and we obtain
\[
\D=c\circ(\nablal^{\mathcal{S}^+}+\nablal^{\mathcal{S}^-})=c\circ \nablal^{\mathcal{S}}.\qedhere
\]
\end{proof}
Next we must check compatibility of the connections, 
meaning that for a one form $\omega$ and spinor $s$ we have
\[
\nablal^{\mathcal{S}}(c(\omega\ox s))=c\circ m\circ(\sigma\ox1)\big(\nablar^{G}(\omega)\ox s +\omega\ox \nablal^{\mathcal{S}}(s)\big)
\]
where $\nablar^G$ is the right Levi-Civita connection on one forms. Here $\sigma$ is the generalised braiding from Definition \ref{defn:ciggy-butt-brain}.

We will compute connections and their squares on the module $\mathcal{S}$, prove the compatibility with the connection on $\Omega^1_\D$ and then prove the Weitzenbock formula. In all of these tasks we use the formulae
\begin{equation}
\partial_b( t^{1/2}_{1/2-k,1/2}t^{1/2\ast}_{1/2-h,1/2})=\left\{\begin{array}{ll} -q^{1/2} t^{1/2}_{1/2-k,1/2}t^{1/2\ast}_{1/2-h,-1/2} & b=e\\
 q^{-1/2}t^{1/2}_{1/2-k,-1/2}t^{1/2\ast}_{1/2-h,1/2} & b=f\end{array}\right.,
\label{eq:half-derivs-again}
\end{equation}
\begin{equation}
\partial_b( t^{1/2}_{1/2-k,-1/2}t^{1/2\ast}_{1/2-h,-1/2})=\left\{\begin{array}{ll} q^{1/2} t^{1/2}_{1/2-k,1/2}t^{1/2\ast}_{1/2-h,-1/2} & b=e\\
 -q^{-1/2}t^{1/2}_{1/2-k,-1/2}t^{1/2\ast}_{1/2-h,1/2} & b=f\end{array}\right.
\label{eq:neg-half-derivs}
\end{equation}
and
\begin{equation}
\partial_e(t^{1/2}_{i,1/2})=0,\quad \partial_e(t^{1/2}_{i,-1/2})=t^{1/2}_{i,1/2},\quad \partial_f(t^{1/2}_{i,1/2})=t^{1/2}_{i,-1/2},\quad \partial_f(t^{1/2}_{i,-1/2})=0.
\label{eq:derivs}
\end{equation}
\[
\partial_e\partial_f(t^{1/2}_{i1/2})=t^{1/2}_{i1/2}
,\qquad
\partial_f\partial_e(t^{1/2}_{i,-1/2})=t^{1/2}_{i,-1/2}
\]
\[
\partial_e\partial_f(t^{1/2}_{i1/2}(t^{1/2}_{l1/2})^*)
=q^{-1}t^{1/2}_{i1/2}(t^{1/2}_{l1/2})^*-qt^{1/2}_{i1/2}(t^{1/2}_{l,-1/2})^*.
\]
Using these formulae we can obtain the Grassmann connections on $\mathcal{S}^\pm$.
\begin{lemma}
The left Grassmann connections $\nablal^{\mathcal{S}^-}$ and $\nablal^{\mathcal{S}^+}$ are given by
\begin{align*}
&\nablal^{\mathcal{S}^+}\begin{pmatrix}b_is_{i,+}\\0\end{pmatrix}
=[\D,b_i]\ox \begin{pmatrix}s_{i,+}\\0\end{pmatrix}+b_i[\D,t_{i,1/2}^{1/2}(t_{l,1/2}^{1/2})^*]\ox \begin{pmatrix}s_{l,+}\\0\end{pmatrix}\\
&=\begin{pmatrix} 0 & q^{-1/2}\partial_e(b_i)\\q^{1/2}\partial_f(b_i)&0\end{pmatrix}\ox\begin{pmatrix}s_{i,+}\\0\end{pmatrix}+b_i\begin{pmatrix} 0 & -t_{i,1/2}^{1/2}(t_{l,-1/2}^{1/2})^*\\t_{i,-1/2}^{1/2}(t_{l,1/2}^{1/2})^*&0\end{pmatrix}\ox \begin{pmatrix}s_{l,+}\\0\end{pmatrix}
\end{align*}
and
\begin{align*}
&\nablal^{\mathcal{S}^-}\begin{pmatrix}0\\c_ls_{l,-}\end{pmatrix}
=[\D,c_l]\ox \begin{pmatrix}0\\s_{l,-}\end{pmatrix}+c_l[\D,t_{l,-1/2}^{1/2}(t_{p,-1/2}^{1/2})^*]\ox \begin{pmatrix}0\\s_{p,-}\end{pmatrix}\\
&=\begin{pmatrix} 0 & q^{-1/2}\partial_e(c_l)\\ q^{1/2}\partial_f(c_l)&0\end{pmatrix}\ox\begin{pmatrix}0\\s_{l,-}\end{pmatrix}+c_l\begin{pmatrix} 0 & t_{l,1/2}^{1/2}(t_{p,-1/2}^{1/2})^*\\ -t_{l,-1/2}^{1/2}(t_{p,1/2}^{1/2})^*&0\end{pmatrix}\ox \begin{pmatrix}0\\s_{p,-}\end{pmatrix}
\end{align*}
\end{lemma}
Using the formulae for the connections we can check condition 4 of Definition \ref{def: Dirac-spectral-triple}.
\begin{thm}
For $[\D,b]a\in \Omega^1_\D$ we have the compatibility
\begin{align}
&\D\Big([\D,b]a\begin{pmatrix}b_is_{i,+}\\c_ls_{l,-}\end{pmatrix}\Big)=c\Big(\nablal^{\mathcal{S}}\Big([\D,b]a\begin{pmatrix}b_is_{i,+}\\c_ls_{l,-}\end{pmatrix}\Big)\Big)\nonumber\\
&=c\Big((m\circ \sigma)\ox1\Big([\D,b]a\ox\nablal^{\mathcal{S}}\begin{pmatrix}b_is_{i,+}\\c_ls_{l,-}\end{pmatrix}\Big)\Big)+c\Big((m\circ \sigma)\ox1\Big(\nablar^G([\D,b]a)\ox\begin{pmatrix}b_is_{i,+}\\c_ls_{l,-}\end{pmatrix}\Big)\Big).
\label{eq:twisted-Leibniz}
\end{align}
\end{thm}
\begin{proof}

Let $b,a\in\B$ and define a one-form by
\[
[\D,b]a=\begin{pmatrix} 0&q^{-1/2}\partial_e(b)a\\q^{1/2}\partial_f(b)a&0\end{pmatrix}.
\]
We compute the left hand side of Equation \eqref{eq:twisted-Leibniz}, ignoring elements of $\ker(m)\subset T^2_\D(\B)$ as they act by zero on the spinor bundle. We obtain
\begin{align*}
&c\left(\nablal^{\mathcal{S}}\left([\D,b]a\begin{pmatrix}b_is_{i,+}\\c_ls_{l,-}\end{pmatrix}\right)\right)\\
&=\begin{pmatrix} q\partial_e\partial_f(b)a&\hspace{-20pt}0\\0&\hspace{-10pt}q^{-1}\partial_e\partial_f(b)a\end{pmatrix}
\begin{pmatrix}b_is_{i,+}\\c_ls_{l,-}\end{pmatrix}
+\begin{pmatrix}q^{2}\partial_f(b)\partial_e(a)b_is_{i,+}+q^{2}\partial_f(b)a\partial_e(b_i)s_{i,+}\\q^{-2}\partial_e(b)\partial_f(a)c_ls_{l,-}+q^{-2}\partial_e(b)a\partial_f(c_l)s_{l,-}\end{pmatrix},
\end{align*}
the last line following since $\partial_f(t_{l,-1/2})=\partial_e(t_{i,1/2})=0$.
Again omitting elements of $\ker(m)$ (which we recall is preserved by $\sigma$), the second term on the right hand side of Equation \eqref{eq:twisted-Leibniz} is
\begin{align*}
&c\circ (m\circ\sigma)\ox1\left(\nablar^G([\D,b]a)\ox\begin{pmatrix}b_is_{i,+}\\c_ls_{l,-}\end{pmatrix}\right)\\
&=c\circ(m\circ\sigma)\ox1\Big(\Big(G \partial_e\partial_f(b)a\begin{pmatrix}b_is_{i,+}\\c_ls_{l,-}\end{pmatrix}\Big)+
\begin{pmatrix}0& \hspace{-5pt}(t^1_{2-j,-1})^*\\0& \hspace{-20pt}0\end{pmatrix}\ox\begin{pmatrix} 0&\hspace{-10pt}0\\t^1_{2-j,-1}&\hspace{-5pt}0\end{pmatrix}\partial_e(b)\partial_f(a)\Big)\\
&+
c\circ(m\circ\sigma)\ox1\left(\begin{pmatrix} 0&0\\(t^1_{2-j,1})^*&0\end{pmatrix}\ox\begin{pmatrix} 0&t^1_{2-j,1}\\0&0\end{pmatrix}\partial_f(b)\partial_e(a)\right)\begin{pmatrix}b_is_{i,+}\\c_ls_{l,-}\end{pmatrix}.
\end{align*}
Applying the braiding $\sigma$ (recalling that $\sigma$ maps $\ker(m)$ to itself) and then implementing the action of forms on $\mathcal{S}$ gives
\begin{align*}
&c\left((m\circ\sigma)\ox1\left(\nablar^G([\D,b]a)\ox\begin{pmatrix}b_is_{i,+}\\c_ls_{l,-}\end{pmatrix}\right)\right)\\
&=c\Big(G \partial_e\partial_f(b)a\begin{pmatrix}b_is_{i,+}\\c_ls_{l,-}\end{pmatrix}\Big)+c\left(q^{-2}\begin{pmatrix} 0&0\\(t^1_{2-j,1})^*&0\end{pmatrix}\ox\begin{pmatrix} 0&t^1_{2-j,1}\\0&0\end{pmatrix}\partial_e(b)\partial_f(a)\right)\\
&+
c\left(q^2\begin{pmatrix} 0&(t^1_{2-j,-1})^*\\0&0\end{pmatrix}\ox\begin{pmatrix} 0&0\\t^1_{2-j,-1}&0\end{pmatrix}\partial_f(b)\partial_e(a)\right)\begin{pmatrix}b_is_{i,+}\\c_ls_{l,-}\end{pmatrix}\\
&=\begin{pmatrix} q \partial_e\partial_f(b)a+q^{2}\partial_f(b)\partial_e(a)&0\\0&q^{-1} \partial_e\partial_f(b)a+q^{-2}\partial_e(b)\partial_f(a)\end{pmatrix}\begin{pmatrix}b_is_{i,+}\\c_ls_{l,-}\end{pmatrix}.
\end{align*}
The first term on the right hand side of Equation \eqref{eq:twisted-Leibniz} is
\begin{align*}
&[\D,b]a\ox\nablal^{\mathcal{S}}\begin{pmatrix}b_is_{i,+}\\c_ls_{l,-}\end{pmatrix}\\
&=\begin{pmatrix}0&q^{-1/2}\partial_e(b)a\\ q^{1/2}\partial_f(b)a&0\end{pmatrix}
\ox
\begin{pmatrix}\hspace{-50pt}0&\hspace{-50pt}q^{-1/2}\partial_e(c_lt_{l,-1/2})t_{p,-1/2}^*\\
q^{1/2}\partial_f(c_lt_{l,-1/2}t_{p,-1/2}^*)&0\end{pmatrix}\ox\begin{pmatrix}0\\s_{p,-}\end{pmatrix}\\
&+\begin{pmatrix}0&q^{-1/2}\partial_e(b)a\\ q^{1/2}\partial_f(b)a&0\end{pmatrix}
\ox\begin{pmatrix} \hspace{-40pt}0&\hspace{-40pt}q^{-1/2}\partial_e(b_it_{i,1/2}t_{k,1/2}^*)\\
 q^{1/2}\partial_f(b_it_{i,1/2})t_{k,1/2}^*&0\end{pmatrix}\ox\begin{pmatrix}s_{k,+}\\0\end{pmatrix}\\
&= \Bigg(\begin{pmatrix} 0&0\\(t^1_{2-j,1})^*&0\end{pmatrix}\ox\begin{pmatrix} 0&t^1_{2-j,1}\\0&0\end{pmatrix}\partial_f(b)a\partial_e(c_lt_{l,-1/2})t_{p,-1/2}^*\\
&+\begin{pmatrix} 0&(t^1_{2-j,-1})^*\\0&0\end{pmatrix}\ox\begin{pmatrix} 0&0\\t^1_{2-j,-1}&0\end{pmatrix}\partial_e(b)a\partial_f(c_l)t_{l,-1/2}t_{p,-1/2}^*\Bigg)\ox\begin{pmatrix}0\\s_{p,-}\end{pmatrix}\\
&+\Bigg(\begin{pmatrix} 0&0\\(t^1_{2-j,1})^*&0\end{pmatrix}\ox\begin{pmatrix} 0&t^1_{2-j,1}\\0&0\end{pmatrix}\partial_f(b)a\partial_e(b_i)t_{i,1/2}t_{k,1/2}^*\\
&+\begin{pmatrix} 0&(t^1_{2-j,-1})^*\\0&0\end{pmatrix}\ox\begin{pmatrix} 0&0\\t^1_{2-j,-1}&0\end{pmatrix}\partial_e(b)a\partial_f(b_it_{i,1/2})t_{k,1/2}^*\Bigg)\ox\begin{pmatrix}s_{k,+}\\0\end{pmatrix}
\end{align*}
where simplification occurs in two terms due to orthogonality, the derivation rule, and $\partial_f(t_{l,-1/2})=\partial_e(t_{i,1/2})=0$.
Applying $\sigma\ox1$ yields
\begin{align*}
&\sigma\ox1\left([\D,b]a\ox\nablal^{\mathcal{S}}\begin{pmatrix}b_is_{i,+}\\c_ls_{l,-}\end{pmatrix}\right)\\
&= \Bigg(q^2 \begin{pmatrix} 0&(t^1_{2-j,-1})^*\\0&0\end{pmatrix}\ox\begin{pmatrix} 0&0\\t^1_{2-j,-1}&0\end{pmatrix}q^{1/2}\partial_f(b)aq^{-1/2}\partial_e(c_lt_{l,-1/2})t_{p,-1/2}^*\\
&+q^{-2}\begin{pmatrix} 0&0\\(t^1_{2-j,1})^*&0\end{pmatrix}\ox\begin{pmatrix} 0&t^1_{2-j,1}\\0&0\end{pmatrix}q^{-1/2}\partial_e(b)aq^{1/2}\partial_f(c_l)t_{l,-1/2}t_{p,-1/2}^*\Bigg)\ox\begin{pmatrix}0\\s_{p,-}\end{pmatrix}\\
&+\Bigg(q^2 \begin{pmatrix} 0&(t^1_{2-j,-1})^*\\0&0\end{pmatrix}\ox\begin{pmatrix} 0&0\\t^1_{2-j,-1}&0\end{pmatrix}q^{1/2}\partial_f(b)aq^{-1/2}\partial_e(b_i)t_{i,1/2}t_{k,1/2}^*\\
&+q^{-2}\begin{pmatrix} 0&0\\(t^1_{2-j,1})^*&0\end{pmatrix}\ox\begin{pmatrix} 0&t^1_{2-j,1}\\0&0\end{pmatrix}q^{-1/2}\partial_e(b)a\partial_f(b_it_{i,1/2})t_{k,1/2}^*\Bigg)\ox\begin{pmatrix}s_{k,+}\\0\end{pmatrix}
\end{align*}
and then applying the action gives
\[
c\left((m\circ\sigma)\ox1\left([\D,b]a\ox\nablal^{\mathcal{S}}\begin{pmatrix}b_is_{i,+}\\c_ls_{l,-}\end{pmatrix}\right)\right)=\begin{pmatrix}0\\q^{-2}\partial_e(b)a\partial_f(c_l)s_{l,-}\end{pmatrix}
+\begin{pmatrix}q^2\partial_f(b)a\partial_e(b_i)s_{i,+}\\0\end{pmatrix}.
\]
Combining these calculations yields the result.
\end{proof}
\begin{thm} 
The spectral triple $(\B,L^{2}(S,h),\D)$ is a Dirac spectral triple relative to $(\Omega^{1}_{\D}(\B),\dag,\pairing{\cdot}{\cdot},\sigma)$. The connection Laplacian $\Delta^{\mathcal{S}}$ of Equation \eqref{eq:conn-lap} and the Dirac operator $\D$ satisfy the Weitzenbock formula
\[
\D^2=\Delta^{\mathcal{S}}+c\circ\sigma(R^{\nablal^{\mathcal{S}}}).
\]
\end{thm}
\begin{proof} We have verified the conditions of Definition \ref{def: Dirac-spectral-triple}, so $(\B,L^{2}(S,h),\D)$ is a Dirac spectral triple over the braided Hermitian differential structure $(\Omega^{1}_{\D}(\B),\dag,\pairing{\cdot}{\cdot},\sigma)$. Therefore the connection Laplacian $\Delta^{\mathcal{S}}:\mathcal{S}\to L^{2}(S,h)$ of Equation \eqref{eq:conn-lap} is a well-defined operator. By Theorem \ref{thm:Weitz1}, the operators $\D$ and $\Delta^{\mathcal{S}}$ are related via the Weitzenbock formula
\[\D^{2}=\Delta^{\mathcal{S}}+c\circ m\circ\sigma\ox 1(R^{\nablal^{\mathcal{S}}}).\qedhere\]
\end{proof}
\subsection{Curvature of the spinor bundle and positivity of the Laplacian}
We now further analyse the Weitzenbock formula for the Podle\'s sphere. Our first goal is to compute the curvature term $c\circ m\circ\sigma\ox 1(R^{\nablal^{\mathcal{S}}})$.
\begin{prop}
The curvature of $\mathcal{S}^+\oplus \mathcal{S}^-$ is given by
\begin{equation}
\label{eq:spinorcurvature}
R^{\nablal^\mathcal{S}}\begin{pmatrix}b_is_{i,+}\\c_ls_{l,-}^{1/2}\end{pmatrix}
=\frac{q}{2}C\ox \begin{pmatrix} -q^{-1}&0\\0&q\end{pmatrix}\begin{pmatrix}b_is_{i,+}\\c_ls_{l,-}\end{pmatrix},
\end{equation}
and its Clifford representation is given by $$c\circ m\circ\sigma\ox 1(R^{\nablal^{\mathcal{S}}})=\frac{1}{q^{-2}+q^2}\begin{pmatrix} q^{2}&0\\0&q^{-2}\end{pmatrix}.$$
\end{prop}
\begin{proof}
To compute the curvature of $\nablal^{\mathcal{S}}$ we 
first observe that the computation
\begin{align*}
1\ox\nablal^{\mathcal{S}^-}([\D,c_lt_{l,-1/2}t_{m,-1/2}^*]\ox t_{m,-1/2})
&=[\D,c_lt_{l,-1/2}t_{m,-1/2}^*]\ox[\D,t_{m,-1/2}t_{p,-1/2}^*]\ox t_{p,-1/2}\\
&=c_l[\D,t_{l,-1/2}t_{m,-1/2}^*]\ox[\D,t_{m,-1/2}t_{p,-1/2}^*]\ox t_{p,-1/2}
\end{align*}
shows that the $[\D,c_l]$ term vanishes because it is multiplied by $v^*[\D,p]v=v^*p[\D,p]pv=0$ where $v$ is the stabilisation map defined by the frame, and $p=vv^*$. The same applies to the $\nablal^{\mathcal{S}^+}$ part of the computation, and we use this freely below.
So we start by computing (throwing away elements of $\ker(m)$ 
and the $[\D,b_i]$ term at the last step)
\begin{align*}
&1\ox\nablal^{\mathcal{S}^+}\circ\nablal^{\mathcal{S}^+}\begin{pmatrix}b_is_{i,+}\\0\end{pmatrix}
=\begin{pmatrix} 0 & q^{-1/2}\partial_e(b_i)\\q^{1/2}\partial_f(b_i)&0\end{pmatrix}\ox\nablal^{\mathcal{S}^+}\begin{pmatrix}s_{i,+}\\0\end{pmatrix}\\
&+\begin{pmatrix} 0 & -b_it_{i,1/2}^{1/2}(t_{k,-1/2}^{1/2})^*\\b_it_{i,-1/2}^{1/2}(t_{k,1/2}^{1/2})^*&0\end{pmatrix}\ox \nablal^{\mathcal{S}^+}\begin{pmatrix}s_{k,+}\\0\end{pmatrix}\\
&=\begin{pmatrix} 0 & q^{-1/2}\partial_e(b_i)\\q^{1/2}\partial_f(b_i)&0\end{pmatrix}\ox\begin{pmatrix} 0 & -t_{i,1/2}^{1/2}(t_{m,-1/2}^{1/2})^*\\t_{i,-1/2}^{1/2}(t_{m,1/2}^{1/2})^*&0\end{pmatrix}\ox \begin{pmatrix}s_{m,+}\\0\end{pmatrix}\\
&+\begin{pmatrix} 0 & -b_it_{i,1/2}^{1/2}(t_{l,-1/2}^{1/2})^*\\b_it_{i,-1/2}^{1/2}(t_{l,1/2}^{1/2})^*&0\end{pmatrix}\ox\begin{pmatrix} 0 & -t_{l,1/2}^{1/2}(t_{m,-1/2}^{1/2})^*\\t_{l,-1/2}^{1/2}(t_{m,1/2}^{1/2})^*&0\end{pmatrix}\ox \begin{pmatrix}s_{m,+}\\0\end{pmatrix}\\
&=\begin{pmatrix}0&(t^1_{2-j,-1})^*\\0&0\end{pmatrix}\ox\begin{pmatrix}0&0\\t^1_{2-j,-1}&0\end{pmatrix}\ox\begin{pmatrix}-b_is_{i,+}\\0\end{pmatrix}.
\end{align*}
Ignoring the kernel of $m$ again (which is orthogonal to the image of $1-\Psi$), using $\B$-linearity of the tensor product and the orthogonality relations gives
\begin{align*}
\nablar^{G}\ox1\circ\nablal^{\mathcal{S}^+}\begin{pmatrix}b_is_{i,+}\\0\end{pmatrix}&=\nablar^{G}([\D,b_it_{i,1/2}^{1/2}(t_{k,1/2}^{1/2})^*])\ox \begin{pmatrix}s_{k,+}\\0\end{pmatrix}\\
&=G\partial_e\partial_f(b_it_{i,1/2}^{1/2}(t_{k,1/2}^{1/2})^*)\ox \begin{pmatrix}s_{k,+}\\0\end{pmatrix}.
\end{align*}
Recalling that $\pairing{ C}{G}=0$, we see that this term does not contribute to the curvature, and applying $1-\Psi$ yields
\begin{align*}
R^{\nablal^{\mathcal{S}^+}}\begin{pmatrix}b_is_{i,+}\\0\end{pmatrix}
&=\frac{1}{\alpha}C\left\langle C,\begin{pmatrix}0&(t^1_{2-j,-1})^*\\0&0\end{pmatrix}\ox\begin{pmatrix}0&0\\t^1_{2-j,-1}&0\end{pmatrix}\right\rangle\ox\begin{pmatrix}-b_is_{i,+}\\0\end{pmatrix}\\
&=-\frac{1}{2}C\ox\begin{pmatrix}b_is_{i,+}\\0\end{pmatrix}.
\end{align*}
We continue with the negative spinors, computed much as the positive ones, and find
\begin{align*}
&1\ox\nablal^{\mathcal{S}^-}\circ\nablal^{\mathcal{S}^-}\begin{pmatrix}0\\c_ls_{l,-}\end{pmatrix}\\
&=c_l\begin{pmatrix} 0 & t_{l,1/2}^{1/2}(t_{m,-1/2}^{1/2})^*\\-t_{l,-1/2}^{1/2}(t_{m,1/2}^{1/2})^*&0\end{pmatrix}\ox \begin{pmatrix} 0 & t_{m,1/2}^{1/2}(t_{p,-1/2}^{1/2})^*\\-t_{m,-1/2}^{1/2}(t_{p,1/2}^{1/2})^*&0\end{pmatrix}\ox \begin{pmatrix} 0\\s_{p,-}\end{pmatrix}\\
\end{align*}
along with
\begin{align*}
&\nablar^{G}\ox1\circ\nablal^{\mathcal{S}^-}\begin{pmatrix}0\\c_ls_{l,-})\end{pmatrix}
=G\ox\begin{pmatrix}0\\ \partial_f\partial_e(c_l)s_{l,-}+q^{3/2}\partial_f(c_l)s_{l,+}+q^{1}c_ls_{l,-}\end{pmatrix}.
\end{align*}
Applying $1-\Psi$ and combining the $\mathcal{S}^-$ and $\mathcal{S}^+$ computations yields \eqref{eq:spinorcurvature}. Finally, observe that the action of $\sigma(C)$ is given by
\[
\frac{-2q^{-1}}{q^2+q^{-2}}\begin{pmatrix}q^{3}&0\\0&-q^{-3}\end{pmatrix}
\]
and so 
\[
c\circ\sigma\ox1(R^{\nablal^{\mathcal{S}}})=c\Big(\sigma(C)\ox\begin{pmatrix}-1/2&0\\0&q^2/2\end{pmatrix}\Big)=\frac{1}{q^2+q^{-2}}\begin{pmatrix}q^2&0\\0&q^{-2}\end{pmatrix},
\]
which completes the proof.
\end{proof}

Using that 
\begin{align*}
&\D^2\begin{pmatrix} b_is_{i,+}\\c_ls_{l,-}\end{pmatrix}
=\begin{pmatrix}q\partial_e\partial_f(b_i)s_{i,+}+q^{-3/2}\partial_e(b_i)s_{i,-}+b_is_{i,+}\\q^{-1}\partial_f\partial_e(c_l)s_{l,-}+q^{3/2}\partial_f(c_l)s_{l,+}+c_ls_{l,-}\end{pmatrix},
\end{align*}
the Weitzenbock formula now yields the explicit expression
\begin{align*}\Delta^{\mathcal{S}}&\begin{pmatrix} b_is_{i,+}\\c_ls_{l,-}\end{pmatrix}=\frac{1}{q^2+q^{-2}}\begin{pmatrix}q^{-2}&0\\0&q^{2}\end{pmatrix}\begin{pmatrix} b_is_{i,+}\\c_ls_{l,-}\end{pmatrix}+\begin{pmatrix}q\partial_e\partial_f(b_i)s_{i,+}+q^{-3/2}\partial_e(b_i)s_{i,-}\\q^{-1}\partial_f\partial_e(c_l)s_{l,-}+q^{3/2}\partial_f(c_l)s_{l,+}\end{pmatrix},
\end{align*}
for the connection Laplacian. Finally we show that the Laplacian is positive, completing the analogy with the usual Lichnerowicz formula on the sphere.
\begin{prop}
\label{prop:Lap-pos}
The Laplacian $\Delta^{\mathcal{S}}$ is positive. 
\end{prop}
\begin{proof}
By \cite[Section 4.4]{MRC}, the result follows provided that for all $x,y\in S$, the divergence term $h(\pairing{\nablar^{G}({}_{\Omega^1}\pairing{\nablal^{\mathcal{S}} x}{y})}{G})$ vanishes where $h$ is the Haar state of $SU_q(2)$.

First we compute
\begin{align*}
&\pairing{\nablar^{G}({}_{\Omega^1}\pairing{\nablal^{\mathcal{S}} x}{y})}{G}\\
&=\Tr\Big(\begin{pmatrix} q&0\\0&q^{-1}\end{pmatrix}\Big(m\circ\nablar^{G}(\pairing{\nablal^{\mathcal{S}} x}{y})-m\circ\nablar^{G}(\omega_j)\pairing{\omega_j}{\pairing{\nablal^{\mathcal{S}} x}{y}}_\B\Big)\Big),
\end{align*}
and so it suffices to show that for all one-forms $\omega$ we have
\begin{equation}
h\Big(\Tr\Big(\begin{pmatrix} q&0\\0&q^{-1}\end{pmatrix}\Big(m\circ\nablar^{G}(\omega)\Big)\Big)\Big)=0.
\label{eq:div}
\end{equation}
Using \eqref{eq:right} we find that
\[
m\circ\nablar^{G}([\D,b]a)=\begin{pmatrix} q&0\\0&q^{-1}\end{pmatrix}
\partial_e\partial_f(b)a+\begin{pmatrix} \partial_e(b)\partial_f(a)&0\\0&\partial_f(b)\partial_e(a)\end{pmatrix}.
\]
Since $\partial_e\partial_f(b)=\partial_f\partial_e(b)$ for all $b\in\B$ we have
\[
\partial_e\partial_f(b)a=\partial_e(\partial_f(b)a)-q\partial_f(b)\partial_e(a)
=\partial_f(\partial_e(b)a)-q^{-1}\partial_e(b)\partial_f(a).
\]
Thus
\[
m\circ\nablar^{G}([\D,b]a)=\begin{pmatrix} q\partial_f(\partial_e(b)a)&0\\0&q^{-1}\partial_e(\partial_f(b)a)\end{pmatrix}
\]
and so
\[
h\Big(\Tr\Big(\begin{pmatrix} q&0\\0&q^{-1}\end{pmatrix}\Big(m\circ\nablar^{G}([\D,b]a)\Big)\Big)\Big)
=q^2h(\partial_f(\partial_e(b)a))+q^{-2}h(\partial_e(\partial_f(b)a)).
\]
Finally, $h(t^\ell_{ij})=\delta_{\ell,0}$, and so $h(\partial_e(c))=h(\partial_f(c))=0$ for all $c\in\B$, and hence the divergence
\eqref{eq:div} vanishes, and \cite[Corollary 4.9]{MRC} gives the positivity of $\Delta^{\mathcal{S}}$.
\end{proof}

\begin{rmk}
For this example the coincidence between the action of the spinor curvature and $1/4$
of the scalar curvature (which classically is $r=2$) holds only in the classical limit $q=1$.
\end{rmk}

\appendix

\section{Uniqueness of Levi-Civita connection}

In this appendix we present the uniqueness condition of Hermitian torsion-free $\sigma$-$\dag$-bimodule connections in the context of Grassmann connections associated to an exact frame. The discussion of \cite{MRLC} simplifies substantially in this setting, and in particular becomes more algebraic.

To  prove the uniqueness of the Hermitian torsion-free connection, we need to accomplish two tasks. First we need to prove the differential structure is concordant, which we describe and prove next. Then we need to verify an injectivity condition, stated below in Theorem \ref{thm:uniqueness}.

\begin{defn}
\label{ass:concordance}
Let $(\Omega^{1}_{\D}(\B),\dag,\Psi,\pairing{\cdot}{\cdot})$ be an Hermitian differential structure.
Define the projections $P:=\Psi\ox1$ and $Q:=1\ox\Psi$ on $T^{3}_{\D}(\B)$. The differential structure is concordant if
$T^{3}_{\D}=({\rm Im}(P)\cap{\rm Im}(Q))\oplus ({\rm Im} (1-P) + {\rm Im} (1-Q))$. Let $\Pi$ be the projection onto ${\rm Im}(P)\cap{\rm Im}(Q)$.
\end{defn}
Strictly speaking the Definition is stated for one-forms over a local algebra, and the polynomial algebra $\mathcal{O}(S^2_q)$ is not local. The completion of $\B$ in the norm $\B\ni b\mapsto\Vert b\Vert+\Vert[\D,b]\Vert$ is local, and we can likewise complete the module of one-forms. We will ignore this distinction below for notational simplicity.

Lemma \ref{lem:kernel} and Proposition \ref{prop:all-two-forms} give us a frame for the two-junk module
\[
Z=\frac{1}{\sqrt{q^2+q^{-2}}}G,\quad X_k,\quad Y_k,\quad k=-2,-1,0,1,2.
\]
The three projections on two-forms $P_Z=\ketbra{Z}{Z}$,
$P_X=\sum_k\ketbra{X_k}{X_k}$ and $P_Y=\sum_k\ketbra{Y_k}{Y_k}$ are mutually orthogonal
and the projection onto the junk submodule is $\Psi=P_Z+P_X+P_Y$.

\begin{lemma}
\label{lem:sub-projns}
The projections $P_Z,P_X,P_Y$ satisfy
\[
P_X\ox1\circ1\ox P_Y=P_Y\ox1\circ 1\ox P_X=0,
\]

\[
P_X\ox1\circ 1\ox P_Z\circ P_Z\ox1\circ 1\ox P_Y
=P_Y\ox1\circ 1\ox P_Z\circ P_Z\ox1\circ 1\ox P_X=0
\]
\[
P_X\ox1\circ 1\ox P_X\circ(\Psi\ox1\circ1\ox\Psi-P_X\ox1\circ 1\ox P_X-P_Y\ox1\circ 1\ox P_Y)=0
\]
\[
P_Y\ox1\circ 1\ox P_Y\circ(\Psi\ox1\circ1\ox\Psi-P_X\ox1\circ 1\ox P_X-P_Y\ox1\circ 1\ox P_Y)=0
\]
\[
P_Y\ox1\circ 1\ox P_Y\circ P_Z\ox1=P_X\ox1\circ 1\ox P_X\circ P_Z\ox1=0
\]
\[
P_X\ox1\circ 1\ox P_Z\circ  P_Y\ox1=P_Y\ox1\circ 1\ox P_Z\circ P_X\ox1=0
\]
\[
P_Z\ox1\circ1\ox P_Z\circ P_X\ox1=P_Z\ox1\circ1\ox P_Z\circ P_Y\ox1=0
\]
\[
(P_X\ox1\circ 1\ox P_X)^2=P_X\ox1\circ 1\ox P_X,\qquad
(P_Y\ox1\circ 1\ox P_Y)^2=P_Y\ox1\circ 1\ox P_Y,
\]
\[
(P_Z\ox1\circ 1\ox P_X)^2=\frac{q^2}{q^2+q^{-2}}(P_Z\ox1\circ 1\ox P_X),\quad (P_Z\ox1\circ 1\ox P_Y)^2=\frac{q^{-2}}{q^2+q^{-2}}(P_Z\ox1\circ 1\ox P_Y)
\]
\[
(P_Z\ox1\circ 1\ox P_Z)^2=\frac{1}{q^2+q^{-2}} P_Z\ox1\circ 1\ox P_Z.
\]
If we write $V_{AB}=P_A\ox1\circ1\ox P_B$ then
\[
\Psi\ox1\circ1\ox\Psi=V_{ZZ}+V_{ZX}+V_{XZ}+V_{ZY}+V_{YZ}+V_{XX}+V_{YY}
\]
and
\begin{align*}
(\Psi\ox1\circ1\ox\Psi)^n&=
\frac{1}{(q^2+q^{-2})^n}\big(V_{ZZ}+q^{2n}(V_{ZX}+V_{XZ})+q^{-2n}(V_{ZY}+V_{YZ})\big)
+V_{XX}+V_{YY}.
\end{align*}
Hence
\[
\Pi=\lim_{n\to\infty}(\Psi\ox1\circ1\ox\Psi)^n=P_X\ox1\circ 1\ox P_X+P_Y\ox1\circ 1\ox P_Y
\]
and the differential structure of the Podle\'{s} sphere is concordant.
\end{lemma}
\begin{proof}
The algebraic relations are all simple if tedious verifications.

To compute the limit $\lim_{n\to\infty}(P_Z\ox1\circ 1\ox P_Z)^n$, we first note that $\pairing{\rho^\dag}{\omega_j}_\B\omega_j^\dag=\rho$, which allows us to compute
\[
P_Z\ox1\circ 1\ox P_Z(\rho\ox\eta\ox\tau)=Ze^{-\beta/2}\ox\rho\pairing{\eta^\dag}{\tau}_\B e^{-\beta/2}. 
\]
Repeating we find
\[
(P_Z\ox1\circ 1\ox P_Z)^2(\rho\ox\eta\ox\tau)=Ze^{-3\beta/2}\ox\rho\pairing{\eta^\dag}{\tau}_\B e^{-3\beta/2}, 
\]
and as $e^{-\beta}=(q^2+q^{-2})^{-1}<1$, we readily see that 
$\lim_{n\to\infty}(P_Z\ox1\circ 1\ox P_Z)^n=0$.
The formula for $(\Psi\ox1\circ1\ox\Psi)^n$ is then just a consequence of the algebraic relations. The concordance is a consequence of \cite[Proposition 4.13]{MRLC}.
\end{proof}

We now turn to the injectivity condition.
In this paper we are concerned with the Hermitian torsion-free $\dag$-bimodule Grassmann connection $\nablar^G$ associated to the exact frame $(\omega_j)$. We rephrase the uniqueness condition of \cite[Theorem 5.14]{MRLC} for this particular case.
In order to do this we require the bimodule isomorphisms
\begin{align}
\alphar&:T^{3}_{\D}\to \overrightarrow{\textnormal{Hom}}^{*}_{\B}(\Omega^1_{\D},T^{2}_{\D}), \quad\alphar(\omega\otimes \eta\ox\tau)(\rho):=\omega\ox\eta\pairing{\tau^{\dag}}{\rho}\nonumber\\
\alphal&:T^{3}_{\D}\to \overleftarrow{\textnormal{Hom}}^{*}_{\B}(\Omega^1_{\D},T^{2}_{\D}), \quad\alphal( \eta\otimes \omega\ox\tau)(\rho):=\pairing{\rho^{\dag}}{\eta}\omega\ox\tau,
\label{eq:alphas}
\end{align}
where $\rho,\eta,\tau,\omega\in \Omega^1_{\D}$.

We denote by $\mathcal{Z}(M)$ the centre of a $\B$-bimodule $M$.
\begin{thm}
\label{thm:uniqueness}
Let  $(\Omega^{1}_{\D}(\B),\dag,\Psi,\pairing{\cdot}{\cdot})$ be the concordant Hermitian differential structure of the Podle\'{s} sphere. 
If the braiding $\sigma:T^{2}_{\D}\to T^{2}_{\D}$ is such that the map 
\[
\alphar+\sigma^{-1}\circ\alphal:\mathcal{Z}({\rm Im}(\Pi))\to\overleftrightarrow{{\rm Hom}}(\Omega^1_\D,T^2_\D)
\]
is injective,

then $\nablar^{G}$ is the unique Hermitian torsion-free $\sigma$-$\dag$-bimodule connection on $\Omega^{1}_{\D}(\B)$. 
\end{thm}
\begin{proof}
Noting that $\dag$-concordance,\cite[Definition 4.30]{MRLC}, is automatic for a Grassmann connection of an exact frame, \cite[Remark 4.31]{MRLC}, the statement follows from \cite[Theorem 5.14]{MRLC}.
\end{proof}

To check the injectivity condition of Theorem \ref{thm:uniqueness}, we first prove a few simple lemmas that help us identify the centre of the module ${\rm Im}(\Pi)$.
\begin{lemma}
The commutant of the Podle\'{s} sphere $\B$ in 
$\A=\mathcal{O}(SU_q(2))$ is the scalar multiples of the identity.
\end{lemma}
\begin{proof}
Every element of $\mathcal{O}(SU_q(2))$ is a linear combination of
$a^ib^j(b^*)^k$ and $(a^*)^ib^j(b^*)^k$ for $i,j,k\in\N$. Up to
scalar multiples, the generators of $\B$ are $bb^*, ab, b^*a^*$. For scalars $\lambda_{ijk},\rho_{ijk}$ we have
\begin{align*}
\big[\sum_{ijk}\lambda_{ijk}a^ib^j(b^*)^k+\rho_{ijk}(a^*)^ib^j(b^*)^k,&bb^*\big]
=\sum_{ijk}[\lambda_{ijk}a^i+\rho_{ijk}(a^*)^i,bb^*]b^j(b^*)^k\\
&=bb^*\sum_{ijk}\big(\lambda_{ijk}(q^{2i}-1)a^i+\rho_{ijk}(q^{-2i}-1)(a^*)^i\big)b^j(b^*)^k.
\end{align*}
The linear independence of the monomials $a^ib^j(b^*)^k$ and $(a^*)^ib^j(b^*)^k$ shows that the sum vanishes only when all coefficients $\lambda_{ijk},\rho_{ijk}$ vanish except perhaps $i=0$. Then consider
\begin{align*}
\Big[\sum_{jk}\lambda_{0jk}b^j(b^*)^k,ab\Big]
&=\sum_{jk}\big[\lambda_{0jk}b^j(b^*)^k,ab\big]=ab\sum_{jk}(q^{-j-k}-1)\lambda_{0jk}b^j(b^*)^k,
\end{align*}
and again the commutator vanishes only when all $\lambda_{0jk}=0$ except perhaps $j=k=0$. Hence only scalars in $\A$ commute with all of $\B$. The same conclusion holds for local and $C^*$-completions by approximation.
\end{proof}
Since $\Omega^1_\D(\B)\subset \mathcal{O}(SU_q(2))^{\oplus2}$ does not intersect the scalars we have
\begin{corl}
The centre of $\Omega^1_\D(\B)$ is $\{0\}$.
\end{corl}

\begin{lemma}
\label{lem:X-Y}
The linear maps $\mathbb{X},\mathbb{Y}:T^{3}_\D\to\Omega^1_\D(\B)$ given on simple tensors $\rho\ox\eta\ox\tau\in T^{3}_\D$ by
\[
\mathbb{X}(\rho\ox\eta\ox\tau)=\sum_{k=-2,-2,0,1,2}(t^2_{k2})^*\pairing{X_k}{\rho\ox\eta}_\B\,\tau
\]
and
\[
\mathbb{Y}(\rho\ox\eta\ox\tau)=\sum_{k=-2,-2,0,1,2}(t^2_{k,-2})^*\pairing{Y_k}{\rho\ox\eta}_\B\,\tau
\]
are bimodule maps, and so map the centre $\mathcal{Z}T^{3}_\D$ to the centre
$\mathcal{Z}\Omega^1_\D(\B)=\{0\}$. 
\end{lemma}
\begin{proof}
We prove the result for $\mathbb{X}$ as the argument for $\mathbb{Y}$ is the same.
Let $\rho\ox\eta\ox\tau$ be a simple tensor and write $\rho=\big(\begin{smallmatrix}0&\rho_{+}\\ \rho_{-}&0\end{smallmatrix}\big)$ and similarly for $\eta,\tau$. For $b\in\B$ we have $\mathbb{X}(\rho\ox\eta\ox\tau b)=\mathbb{X}(\rho\ox\eta\ox\tau)b$ by definition. Using Lemma \ref{lem:kernel} we have
\begin{align*}
\mathbb{X}(b\rho\ox\eta\ox\tau)&=\sum_{k=-2,-2,0,1,2}(t^2_{k2})^*\pairing{X_k}{b\rho\ox\eta}_\B\,\tau=q\sum_{k=-2,-2,0,1,2}(t^2_{k2})^*t_{k2}^2b\rho_-\eta_-\,\tau\\
&=b\rho_-\eta_-\,\tau=b\mathbb{X}(\rho\ox\eta\ox\tau).
\end{align*}
Hence $\mathbb{X}$ is a bimodule map, and so maps centres to centres.
\end{proof}
\begin{corl}
\label{cor:trivcenter}
For $q\in(0,1)$ we have $\mathcal{Z}{\rm Im}(\Pi)=\mathcal{Z}\Pi(T^3_\D)=\{0\}$.
\end{corl}
\begin{proof}
Let $\rho\ox\eta\ox\tau$ be a simple tensor and again write $\rho=\big(\begin{smallmatrix}0&\rho_{+}\\ \rho_{-}&0\end{smallmatrix}\big)$ and similarly for $\eta,\tau$. Then, using the orthogonality relations, the projection onto ${\rm Im}(\Pi)$ is
\begin{align*}
\Pi(\rho\ox\eta\ox\tau)
&=
\begin{pmatrix} 0&0\\ \rho_{-}&0\end{pmatrix}\ox\begin{pmatrix} 0&0\\ \eta_{-}&0\end{pmatrix}\ox\begin{pmatrix} 0&0\\ \tau_{-}&0\end{pmatrix}
+
\begin{pmatrix} 0&\rho_{+}\\ 0&0\end{pmatrix}\ox\begin{pmatrix} 0&\eta_{+}\\ 0&0\end{pmatrix}\ox\begin{pmatrix} 0&\tau_{+}\\ 0&0\end{pmatrix}\\
&=\sum_{j,k=-1,0,1}\begin{pmatrix} 0&0\\ (t^1_{j1})^*&0\end{pmatrix}\ox\begin{pmatrix} 0&0\\ t^1_{j1}\rho_{-}\eta_{-}\tau_{-}(t^1_{k,-1})^*&0\end{pmatrix}\ox\begin{pmatrix} 0&0\\ t^1_{k,-1}&0\end{pmatrix}\\
&+
\sum_{j,k=-1,0,1}\begin{pmatrix} 0&(t^1_{j,-1})^*\\ 0&0\end{pmatrix}\ox\begin{pmatrix} 0&t^1_{j,-1}\rho_{+}\eta_{+}\tau_{+}(t^1_{k1})^*\\ 0&0\end{pmatrix}\ox\begin{pmatrix} 0&t^1_{k1}\\ 0&0\end{pmatrix}.
\end{align*}
Hence a linear combination $\Pi(\sum_i\rho^i\ox\eta^i\ox\tau^i)$ is zero if and only if
both $\sum_i\rho^i_-\eta^i_-\tau^i_-=0$ and $\sum_i\rho^i_+\eta^i_+\tau^i_+=0$. These observations together with Lemma \ref{lem:X-Y} prove the claim.
\end{proof}
\begin{thm}
\label{thm:unique-pods}
The $\dag$-bimodule connection $(\nablar^G,\sigma)$ is the unique Hermitian torsion-free $\dag$-bimodule connection on $\Omega^{1}_{\D}(\B)$.
\end{thm}
\begin{proof} We need to show that for $q\in(0,1]$ 
the linear map
$$\alphar+\sigma^{-1}\circ\alphal:\mathcal{Z}{\rm Im}(\Pi)\to \overleftrightarrow{{\rm Hom}}(\Omega^1_\D,T^2_\D),$$ is injective. For $q\in (0,1)$ this follows from Corollary \ref{cor:trivcenter}.  For $q=1$, \cite[Lemma 6.11]{MRLC} shows that $\alphar+\sigma^{-1}\circ\alphal$ is injective.
\end{proof}

\end{document}